\documentclass[10pt,a4paper,reqno]{amsart} 

\usepackage{amsmath}
\usepackage{mathtools}
\usepackage{amssymb}
\usepackage{graphicx}
\usepackage{color}
\usepackage{latexsym}
\usepackage[utf8]{inputenc}
\usepackage[T1]{fontenc}

\newcommand{\zs}{{\mathbb Z}} 
\newcommand{\cs}{{\mathbb C}} 
\newcommand{\rs}{{\mathbb R}} 

\newcommand{\PB}{B}

   \renewcommand\Im{\operatorname{Im}}

\newcommand{\ga}{\gamma}

\newcommand{\bU}{\bar U}
\newcommand{\bu}{\bar u}

\newcommand{\E}{\mathbb{E}}
\newcommand{\Var}{\mathbb{V}}

\newcommand{\wt}{\widetilde w}

\newcommand{\cD}{\mathcal D}
\newcommand{\cI}{\mathcal I}

\newcommand{\cL}{\mathcal L}

\newcommand{\cP}{\mathcal P}
\newcommand{\cE}{\mathcal E}
\newcommand{\cS}{\mathcal S}
\newcommand{\cT}{\mathcal T}

\newtheorem{Theorem}{Theorem}
\newtheorem{Proposition}[Theorem]{Proposition}

\newtheorem{Definition}[Theorem]{Definition}
\newtheorem{Lemma}[Theorem]{Lemma}

\newcommand{\beq}{\begin{equation}}
\newcommand{\eeq}{\end{equation}}

\newcommand{\gf}{generating function}
\newcommand{\gfs}{generating functions}

\newcommand{\saw}{self-avoiding walk}
\newcommand{\saws}{self-avoiding walks}
\def\emm#1,{{\em #1}}

 \def\NN{{\sf N}}
 \def\EE{{\sf E}}
 \def\SS{{\sf S}}
 \def\WW{{\sf W}}

\graphicspath{{Figures/}}

\catcode`\@=11
\def\section{\@startsection{section}{1}%
 \z@{.7\linespacing\@plus\linespacing}{.5\linespacing}%
 {\normalfont\bfseries\scshape\centering}}

\def\subsection{\@startsection{subsection}{2}%
  \z@{.5\linespacing\@plus\linespacing}{.5\linespacing}%
  {\normalfont\bfseries\scshape}}

\def\subsubsection{\@startsection{subsubsection}{3}%
 \z@{.5\linespacing\@plus\linespacing}{-.5em}
  {\normalfont\bfseries\itshape}}
\catcode`\@=12

\def\qed{$\hfill{\vrule height 3pt width 5pt depth 2pt}$}

\begin{document}
\title
[Weakly directed self-avoiding walks]
{Weakly directed self-avoiding walks}

\author[A. Bacher]{Axel Bacher}

\author[M. Bousquet-M\'elou]{Mireille Bousquet-M\'elou}
\address{AB: LaBRI, Universit\'e Bordeaux 1, 
351 cours de la Lib\'eration, 33405 Talence, France}
\email{axel.bacher@labri.fr}
\address{MBM: CNRS, LaBRI, Universit\'e Bordeaux 1, 
351 cours de la Lib\'eration, 33405 Talence, France}
\email{mireille.bousquet@labri.fr}

\thanks{}

\keywords{Enumeration -- Self-avoiding walks}
\subjclass[2000]{05A15}

\begin{abstract}
We define a new family of self-avoiding walks (SAW) on the square
lattice, 
called \emph{weakly directed walks}. These walks have a simple
characterization in terms of the irreducible bridges that compose them.
We determine their generating function.
 This series  has a complex singularity 
structure and in particular, is not D-finite.
The growth constant is approximately
$2.54$ and is thus larger than that of all natural families of 
SAW enumerated so far (but smaller than that of general SAW, which is
about 2.64). We also
prove that the end-to-end distance  of weakly directed walks grows
linearly. Finally, we study a diagonal variant of this model.
\end{abstract}

\date{\today}
\maketitle


\section{Introduction}
A lattice walk is \emm  self-avoiding, if it never visits the
same vertex  twice (Fig. ~\ref{fig:saw}).
 Self-avoiding walks (SAW) have attracted interest for decades, first
 in statistical physics, where they are considered as polymer
models, and then in
combinatorics and in probability theory~\cite{madras-slade}.  However,
their properties remain poorly 
understood
in low dimension, despite the existence of remarkable conjectures.
See~\cite{madras-slade} for dimension 5 and above, 
and~\cite{brydges-slade} for recent progresses in 4 dimensions. 

\vskip -4mm 
\begin{figure}[htb]
\includegraphics[height=3cm]{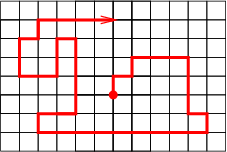}\hskip 20mm
\includegraphics[height=4cm]{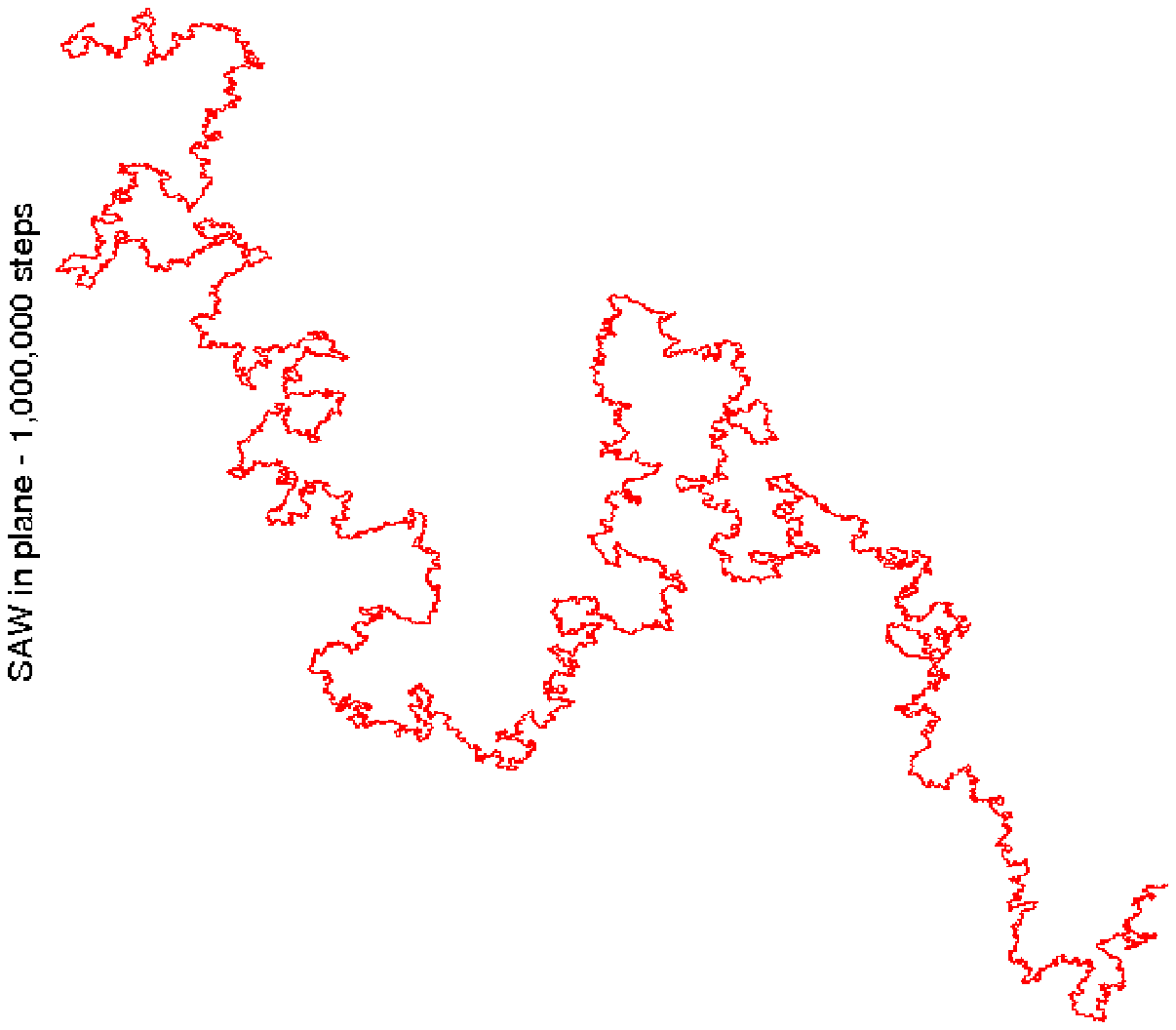}
\caption{A self-avoiding walk on the square lattice, and a
random SAW of length 1,000,000, constructed by Kennedy using a
  pivot algorithm~\cite{kennedy-pivot-SAW}.} 
\label{fig:saw}
\end{figure}

On two-dimensional lattices, it is strongly believed that  the number
$c_n$ of $n$-step SAW and the average end-to-end distance $D_n$ of these
walks satisfy
\beq\label{grail}
c_n\sim \alpha \mu^n n^\gamma \quad \hbox{and} \quad 
D_n \sim \kappa n^{\nu}
\eeq
where $\gamma=11/32$ and $\nu=3/4$. Several independent, but so far not
completely rigorous methods  predict these values, like
numerical studies~\cite{guttmann-conway2001,rechni-buks-saw}, comparisons with other models~\cite{de-gennes,nienhuis82}, probabilistic
arguments involving SLE processes~\cite{lawler-schramm-werner},
enumeration of SAW on random planar lattices~\cite{duplantier-kostov}... The \emm growth constant,
(or \emm connective constant,) $\mu$ is
lattice-dependent. It has recently
 been proved to be $\sqrt{2+\sqrt 2}$ for the 
honeycomb lattice~\cite{duminil-smirnov}, as predicted for almost 30
years, and might be another bi-quadratic number (approximately 
$2.64$) for the square lattice~\cite{jensen-guttmann}.

Given the difficulty of the problem, the study of \emm restricted,
classes of SAW  is natural, and probably as old
as the interest in SAW
itself. The rule of this game is to design  new classes of SAW that have both: 
\begin{itemize}
\item [--] a natural description (to be conceptually pleasant),
\item  [--] some  structure (so that the walks can be counted,
  and their asymptotic properties determined).
\end{itemize}
This paper fits in this program: we define and count a new large class
of SAW, called \emm weakly directed walks,. 

\begin{figure}[htb]
\includegraphics[scale=0.7]{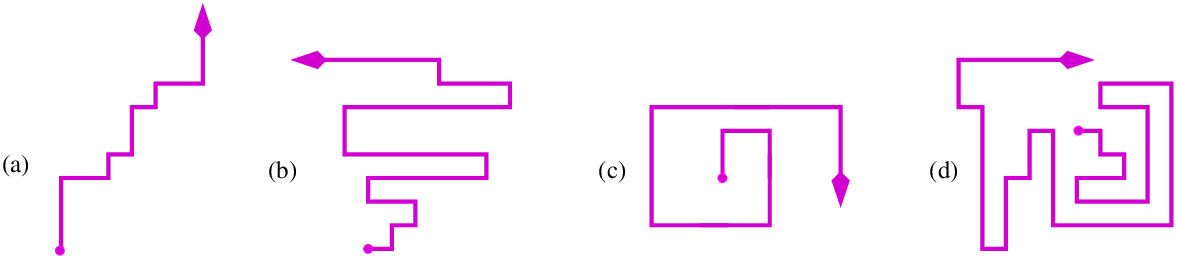}
\caption{(a) A directed walk. (b) A partially directed walk. 
(c) A spiral walk. (d) A prudent walk.} 
\label{fig:restricted}
\end{figure}

The two simplest classes of SAW  on the square lattice probably
consist of \emm directed, and \emm 
partially directed, walks: a walk is directed if it involves at most
two types of steps (for instance North and East), and partially
directed if it involves at most three 
types of steps (Fig.~\ref{fig:restricted}(a-b)).  Partially directed
walks play a prominent role in the definition of our \emm weakly directed, walks. 
Among other solved classes, let us cite  
spiral SAW~\cite{privman-spiral,guttmann-wormald-spiral} and prudent
walks~\cite{mbm-prudent,duchi,guttmann-prudent}. We refer again to  Fig.~\ref{fig:restricted} for illustrations.
Each time such a new class is defined and solved, one
compares its properties to~\eqref{grail}: have we reached with this class
a large growth constant? Is the end-to-end distance of the
walks sub-linear?

At the moment, the largest growth constant (about $2.48$) is obtained
with prudent SAW. However, this is beaten by certain classes whose
description involves a (small) integer $k$, like SAW confined to a strip of
height $k$~\cite{alm-janson,zeilberger-skinny}, or SAW consisting of
\emm irreducible bridges, of length at most
$k$~\cite{jensen-bridges,kesten-bridges}.  The structure of these 
walks is rather poor, which makes them rather unattractive 
from a combinatorial viewpoint. In the former case,  they are described by a
transfer matrix (the size of which increases exponentially
with the height of the strip); in the latter case, the structure is
even simpler, since these walks are just arbitrary  sequences of
irreducible bridges of small length.  
In both cases, the \gf\ is rational.
  The growth constant increases  with $k$, providing better and better
  lower bounds on the growth constant of general SAW. The ability of solving these
  models  for larger values of $k$ mostly relies on
 progress in computer power. 
Regarding asymptotic properties, almost all
solved classes of SAW exhibit a linear end-to-end distance, with the
exception of spiral walks, which are designed so as to wind around
their origin. But there are very few such walks~\cite{guttmann-wormald-spiral}, as their growth
constant is 1.

With  the \emm weakly directed walks, of this paper, we reach a growth
constant of about $2.54$. These walks are defined in the next
section. Their \gf\ is given in Section~\ref{sec:gf}, after some preliminary
results on  partially directed \emm bridges,
(Sections~\ref{sec:culminant-kernel} and
\ref{sec:culminant-heaps}). This series turns out to be much more
complicated that the \gfs\ of  directed and partially directed walks,
which are  rational: we prove 
that it has a natural boundary in the complex plane, and in
particular is not D-finite (that is, it does not satisfy any linear
differential equation with polynomial coefficients).
However, we are able to derive from this series
certain  asymptotic properties of weakly directed walks, like their
growth constant and average end-to-end distance (which we find,
unfortunately, to grow linearly with the length).
Finally, we perform in Section~\ref{sec:diag} a similar study for 
a diagonal variant of weakly directed walks. 
Our  intuition told us that this variant would give a larger growth
constant, but we shall see that this is wrong.
Section~\ref{sec:comments} discusses a few more points, including
random generation.

An extended abstract of this paper appeared in the proceedings of the
2010 FPSAC conference~\cite{bacher-mbm-fpsac}. 

\section{Weakly directed walks: definition}
\label{sec:def}
Let us denote by \NN,  \EE, \SS\ and \WW\ the four square lattice
steps. All walks in this paper are self-avoiding, so that this
precision will often be omitted. 
For any subset $\cS$ of $\{ \NN, \EE,  \SS, \WW\}$, we say that
a (self-avoiding) walk is an $\cS$-\emm walk, if all its steps lie in
$\cS$. For instance, the first walk of Fig.~\ref{fig:restricted} is a
\NN\EE-walk, but also a \NN\EE\WW-walk. The second is  a
\NN\EE\WW-walk.  
We say that a SAW is \emm directed, if it involves at most two types
of steps, and \emm partially directed, if it involves at most three
types of steps. 

The definition of  \emm weakly directed, walks stems from 
the following simple observations:
\begin{itemize}
\item [(i)] between two visits to any given horizontal line, a
  \NN\EE-walk only takes \EE\ steps,
\item [(ii)] between two visits to any given horizontal line, a
  \NN\EE\WW-walk only takes \EE\ and \WW\ steps.
\end{itemize}
Conversely, a walk satisfies (i) if and only if it is either a
\NN\EE-walk or, symmetrically, a \SS\EE-walk. Similarly,  a walk satisfies (ii) if and only if it is either a
\NN\EE\WW-walk or, symmetrically, a \SS\EE\WW-walk. Conditions (i) and
(ii) thus respectively characterize (up to symmetry) \NN\EE-walks and
\NN\EE\WW-walks. 


\begin{Definition}
A walk is \emm weakly directed,
  if, between two visits to any given horizontal line, the walk is
  partially directed (that is, avoids at least one of the steps \NN, \EE, \SS, \WW).
\end{Definition}
Examples are shown in Fig.~\ref{fig:weak-h}.

\begin{figure}[htb]
\scalebox{0.7}{
\input{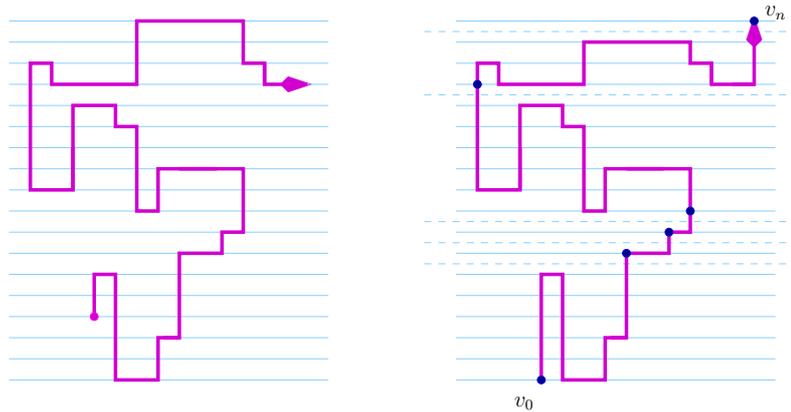}}
\caption{Two weakly directed walks. The second one is a bridge, formed
  of~5 irreducible bridges. Observe
  that these irreducible bridges are partially directed.}
\label{fig:weak-h}
\end{figure}

We will primarily focus on the enumeration of weakly directed \emm
bridges,. As we shall see, this does not affect the growth constant. A
\saw\ starting at $v_0$ and ending at $v_n$ is a \emm  bridge,  if all
its vertices $v\not =v_n$ satisfy $h(v_0)\le h(v) <h(v_n)$, where $h(v)$,
the \emm height, of $v$, is its ordinate. Concatenating two bridges
always gives a bridge. Conversely, every bridge can be uniquely
factored into a sequence of \emm irreducible, bridges (nonempty bridges that
cannot be written as the product of two nonempty bridges). This factorization
is obtained by cutting the walk above each horizontal line of height $n+1/2$
(with $n\in \zs$)
that the walk intersects only once (Fig.~\ref{fig:weak-h}, right). It is known
that the growth constant of bridges is the same as that of general
\saws~\cite{madras-slade}.  
 The fact that bridges
can be  freely concatenated makes them 
useful objects in the study of \saws~\cite{hammersley-welsh,jensen-bridges,kesten-bridges,lawler-schramm-werner,madras-slade}.   

The following result shows that the enumeration of weakly directed
bridges boils down to the enumeration of (irreducible) partially
directed bridges.  
It will be extended to general walks in Section~\ref{sec:gf}.

\begin{Proposition}\label{prop:equiv}
  A bridge is weakly directed if and only if each of its irreducible
  bridges is 
 partially directed (that is, avoids at least one of the steps \NN,
 \EE, \SS, \WW). 
In fact, this means that  each of its irreducible
  bridges is a \NN\EE\SS- or \NN\SS\WW-walk.
\end{Proposition}
\begin{proof}
 The second condition (being \NN\EE\SS\ or \NN\SS\WW) looks more
 restrictive than the first one (being partially directed), but it is
 easy to see 
 that they are actually  equivalent:
 no non-empty \EE\SS\WW-walk is a bridge, and the only irreducible
 bridges among \NN\EE\WW-walks consist of a  sequence of horizontal steps, followed by
 a \NN\ step: thus they are  \NN\EE\SS- or  \NN\SS\WW-walks.

So let us now consider a bridge whose irreducible bridges are partially
  directed. The portion of the walk 
lying  between two visits to a
  given horizontal line is entirely contained in one irreducible
  bridge, and consequently, is partially directed.

Conversely, consider a weakly directed bridge and one of its
irreducible bridges $w$. Of course, $w$ is also weakly directed. Let
$v_0, \ldots, v_n$ be the vertices of $w$, and let $s_i$ be the step
that goes from  $v_{i-1}$ to $v_i$. We
want to prove that $w$ is a \NN\EE\SS- or \NN\SS\WW-walk.
Assume that, on the contrary, $w$ contains a \WW\ step and an \EE\ step. By symmetry, we
may assume  that the first \WW\ occurs before the first \EE.
Let $s_{k+1}$ be the first \EE\ step, and let $s_j$ be the last
\WW\ step before $s_{k+1}$. Then $s_{j+1}, \ldots, s_k$ is a sequence
of \NN\ or \SS\ steps. Let $h$ be the height of $s_{k+1}$.
\begin{itemize}
\item Assume that $s_{j+1}, \ldots, s_k$ are  \NN\  steps
(first walk in Fig.~\ref{fig:preuve-equiv}).
Let $h'$ be the
maximal height reached before $v_j$, say at $v_i$, with $i< j$. Then
$h'<h$ (otherwise, between the first visit to height $h$ 
 and $v_{k+1}$, the walk would not be partially directed). Given that
$w$ is irreducible, it must visit height $h'$ again after $v_{k+1}$,
say at $v_\ell$. But then the walk joining $v_i$ to $v_\ell$ is not
partially directed, a contradiction. 
\item Assume that $s_{j+1}, \ldots, s_k$ are  \SS\  steps
(second walk in Fig.~\ref{fig:preuve-equiv}). Let $v_i$, with $i<k$, be
  the last visit at height $h$ 
  before $v_k$. Then the portion of the walk joining $v_i$ to
  $v_{k+1}$ is not partially directed, a contradiction. 
\end{itemize}
Consequently, the irreducible bridge $w$ is  a \NN \EE\SS- or
\NN\SS\WW-walk.

\begin{figure}[htb]
\scalebox{0.8}{
\input{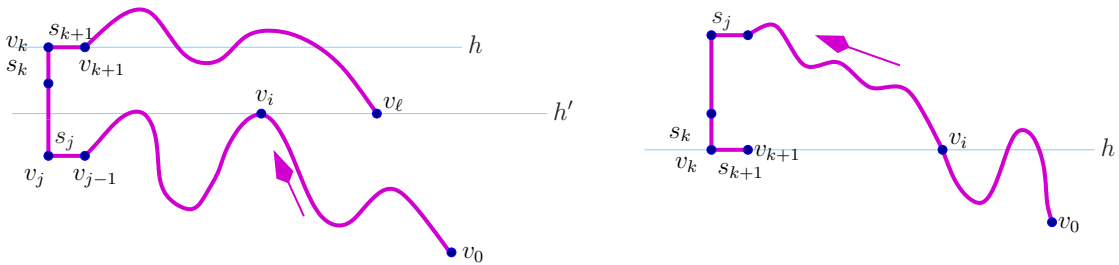}}
\caption{Illustrations for the proof of {Proposition}~\ref{prop:equiv}.}
\label{fig:preuve-equiv}
\end{figure}

\end{proof}

We discuss in Section~\ref{sec:diag}  a variant of weakly directed walks,
where we constrain the walk to be partially directed between two
visits to the same \emm diagonal, line (Fig.~\ref{fig:weak-d}). The notion of
bridges is adapted accordingly, by defining the \emm height, of a
vertex as the
sum of its coordinates. We will refer
to this model as the \emm diagonal model,, and to the original one as the
\emm horizontal model,.  There is, however, no simple counterpart of
Proposition~\ref{prop:equiv}: a (diagonal) bridge whose irreducible
bridges are partially directed is always weakly directed, but the
converse is not true, as can be seen in Fig.~\ref{fig:weak-d}. Thus bridges
with partially directed irreducible bridges form a proper subclass of
weakly directed bridges. We  will enumerate this subclass,
and study its asymptotic properties.

\begin{figure}[htb]
\includegraphics[scale=0.7]{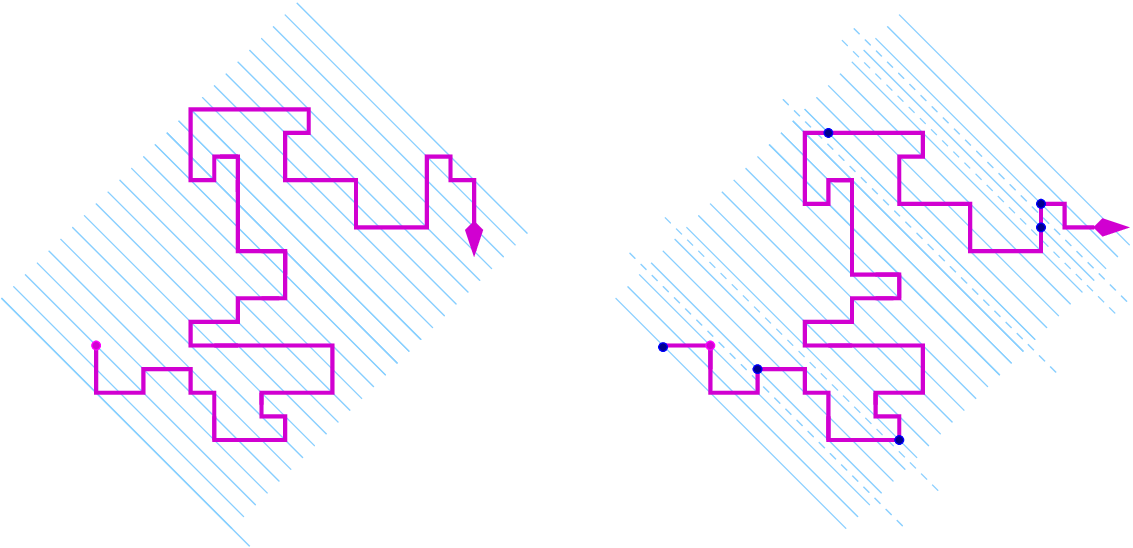}
\caption{Two weakly directed walks in the diagonal model. The second
  one is a bridge, factored into 6 irreducible
  bridges. Observe
  that the third irreducible bridge is \emm not, partially directed.}
\label{fig:weak-d}
\end{figure}

\section{Partially directed bridges: a step-by-step approach}
\label{sec:culminant-kernel}
Let us equip the square lattice $ \zs^2$ with its standard coordinate system.
With each model (horizontal or diagonal) is associated a notion of
\emm height,:  
the height of a vertex $v$, denoted by $h(v)$, is its ordinate in the horizontal model, and the sum of its
coordinates in the diagonal model.
Recall that a walk, starting at $v_0$ and ending at $v_n$, is  a  bridge if  all its
vertices $v\not = v_n$ satisfy $h(v_0) \le h(v) < h(v_n)$. If the
weaker inequality $h(v_0) \le h(v) \le
h(v_n)$ holds for all $v$, we say the walk is a \emm
pseudo-bridge,. Note that nonempty bridges 
are obtained by adding a step of height~$1$ to a
pseudo-bridge (a \NN\ step in the horizontal model, a \NN\ or \EE
\ step in the diagonal model). 
It is thus equivalent to count bridges or pseudo-bridges.

By Proposition~\ref{prop:equiv}, the enumeration of weakly directed
bridges in the horizontal model boils down to the enumeration of (irreducible)
partially directed bridges.  In this section
and the following one, we 
address the enumeration of these building blocks, first in a rather
systematic way based on a step-by-step construction, then in  a more
combinatorial way based on \emm heaps of cycles,. A third approach is
briefly discussed in the final section. We also count
partially directed bridges in the diagonal model, which will be useful
in Section~\ref{sec:diag}.

\begin{figure}[htb]
\includegraphics[height=4cm]{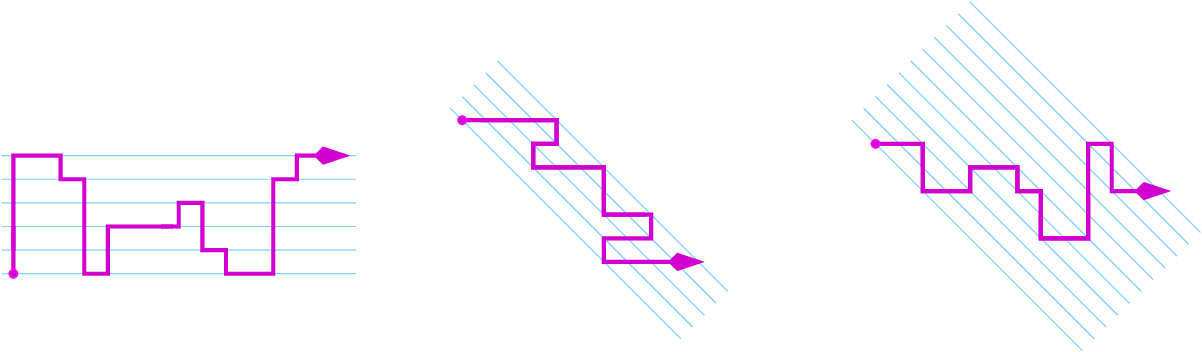}
\caption{A \NN\EE\SS-pseudo-bridge in the horizontal model. (b) 
An \EE\SS\WW-pseudo-bridge in the diagonal model. (c) 
A  \NN\EE\SS-pseudo-bridge in the diagonal model.}
\label{fig:culminants}
\end{figure}

As partially directed walks are defined by the avoidance of (at least)  one step, there
are four  kinds of these. Hence, in principle, we
should count, for each model (horizontal and diagonal), four families
of  partially directed bridges. However, in the horizontal model,
there exists no non-empty \EE\SS\WW-bridge, and every \NN\EE\WW-walk  is a pseudo-bridge. 
The latter class of walks is very easy to count,
and has a rational \gf \ (Lemma~\ref{lem:TPQ}). Moreover, a  
symmetry transforms  \NN\EE\SS-bridges into  
\NN\SS\WW-bridges, so that there is really one class of bridges that
we need to count. 
In the diagonal model, we need to count   \EE\SS\WW-bridges (which are
equivalent to  \NN\SS\WW-bridges by a diagonal symmetry)
and 
\NN\EE\SS-bridges (which are equivalent to \NN\EE\WW-bridges). 
Finally, to avoid certain ambiguities,
we need to count  \EE\SS-bridges, but this has
already been done in~\cite{mbm-ponty}.
 
From now on,  the starting point of our walks is always at height
$0$.  The  height of a walk is then defined to be the
maximal height reached by its vertices.

\subsection{\NN\EE\SS-bridges in the horizontal model}

\begin{Proposition}\label{prop:culm-h}
 Let $k\ge 0$. In  the horizontal model, the length \gf\ of
   \NN\EE\SS-pseudo-bridges of height $k$  is
$$
\PB^{(k)}(t) = \frac{t^k}{G_{k}(t)},
$$
where $G_k(t)$ is the sequence of polynomials defined by
$$
G_{-1} = 1, \quad G_0=1-t, 
\quad \hbox{and for } k\ge 0,
\quad G_{k+1}=(1-t+t^2+t^3)G_k -t^2G_{k-1} .
$$
Equivalently,
\beq\label{pb-inv}
\sum_{k\ge 0} \frac{v^k t^k}{\PB^{(k)}(t) }= \sum_{k\ge 0}v^k G_k=
\frac{1-t-t^2v}{1-(1-t+t^2+t^3)v+t^2v^2},
\eeq
or
$$
\PB^{(k)}(t) = \frac{U-\bU}{\bigl((1-t)U-t\bigr)U^k-\bigl((1-t)\bU-t\bigr)\bU^k},
$$
where 
$$
U= {\frac {1-t+{t}^{2}+{t}^{3}-
\sqrt { \left(1- t^4 \right)   \left(1-2t- {t}^{2}\right)  }}
{2t}}
$$
is a root of 
$t{u}^{2}- \left( 1-t+{t}^{2}+{t}^{3} \right) u+t=0$ 
and $\bU:=1/U$ is the other root of this polynomial.
\end{Proposition}
\begin{proof}
Fix $k\ge 0$.
Let $\cT$ be the set of \NN\EE\SS-walks that end with an \EE\ step,
  and in which each vertex $v$ satisfies $0\le h(v)\le k$. Let $\cT_i$
  be the subset of $\cT$ consisting of walks that end at height
  $i$. Let $T_i(t)\equiv T_i$ be the length \gf\ of  $\cT_i$, and
  define  the bivariate \gf
$$
T(t;u)\equiv T(u)= \sum_{i=0}^k T_i(t) u^i.
$$
This series counts walks of $\cT$ by their length and the height of
their endpoint. Note that we often omit the variable $t$ in our notation.
 The walks of $\cT_k$ are obtained by adding an
\EE\ step at the end of a pseudo-bridge of height $k$, and hence $\PB^{(k)}(t)=
T_k(t)/t$. Alternatively, pseudo-bridges  of height $k$ containing
at least one \EE\ step are obtained
by adding a sequence of  \NN\ steps of appropriate length to a walk of
$\cT$, and this gives
\beq\label{C-W-h}
\PB^{(k)}(t)=t^k  + \sum_{i=0}^k T_i(t) t^{k-i}=t^k\left(1+ T(1/t)\right).
\eeq
(The term $t^k$ accounts for the walk formed of $k$ consecutive \NN\ steps.)
\begin{Lemma}\label{lem:culm-h}
  The series $T(t;u)$, denoted $T(u)$ for short,
 satisfies the following equation:
$$
\left(1-\frac{ut^2}{1-tu}- \frac{t}{1-t\bu}\right) T(u)
=
 t\, \frac{1-(tu)^{k+1}}{1-tu}- t\, \frac{(tu)^{k+1}}{1-tu} T(1/t)
- \frac{t^2\bu}{1-t\bu} T(t),
$$
with $\bu=1/u$.
\end{Lemma}
\begin{proof}
We partition the set $\cT$ into three disjoint subsets, illustrated in
Fig.~\ref{fig:culm-h-dec}. 
\begin{itemize}
\item 
The first subset  consists of walks with a single \EE\ step. These walks
 read $\NN\cdots \NN \EE$, with at most $k$ occurrences of \NN,  and
 their \gf\ is  
$$
t\sum_{i=0}^k (tu)^i= t\, \frac{1-(tu)^{k+1}}{1-tu}.
$$
  \item 
The second subset consists of
walks in which the last \EE\ step is strictly higher than the
previous one. Denoting by $i$ the height of the next-to-last
\EE\ step, the \gf\ of this subset reads
$$
t\, \sum_{i=0}^{k}\left(  T_i (t) u^i \sum_{j=1}^{k-i} (tu)^j \right)
=
t\, \sum_{i=0}^{k}\left(  T_i (t) u^i\, \frac{tu-(tu)^{k-i+1}}{1-tu} \right)
=
\frac{ut^2}{1-tu}T(u) - t\, \frac{(tu)^{k+1}}{1-tu} T(1/t).
$$
   \item 
The third subset  consists of
walks in which the last \EE\  step is weakly lower than the
previous one. Denoting by $i$ the height of the next-to-last
\EE\ step,  the \gf\ of this subset reads
$$
t\, \sum_{i=0}^{k} \left(T_i (t) u^i \sum_{j=0}^{i} (t\bu)^j  \right)
=
t\, \sum_{i=0}^{k}\left( T_i (t) u^i \, \frac{1-(t\bu)^{i+1}}{1-t\bu} \right)
=
\frac{t}{1-t\bu}T(u)- \frac{t^2\bu}{1-t\bu} T(t).
$$
 \end{itemize}
Adding the three contributions gives the series $T(u)$ and establishes the lemma.
\end{proof}

\begin{figure}[htb]
\scalebox{0.8}{\input{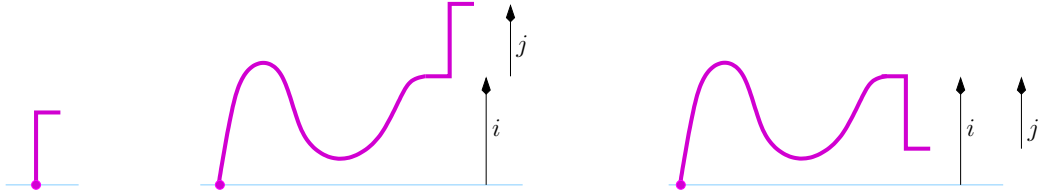}}
\caption{Recursive construction of bounded \NN\EE\SS-walk in the horizontal model.}
\label{fig:culm-h-dec}
\end{figure}

The equation of Lemma~\ref{lem:culm-h} is easily solved using the \emm
kernel method, (see e.g.~\cite{hexacephale,bousquet-petkovsek-1,prodinger}). The \emm kernel, of the equation is the coefficient of $T(u)$, namely
$$
1-\frac{ut^2}{1-tu}- \frac{t}{1-t\bu}.
$$
It vanishes when $u=U$ and $u=\bU:=1/U$, where $U$ is defined in the
lemma.  Since $T(u)$ is a polynomial in $u$, the series $T(U)$ and
$T(\bU)$ are well-defined. Replacing $u$ by $U$ or $\bU$ in the
functional 
equation cancels the left-hand side, and hence the right-hand side. One
thus obtains two linear equations between $T(t)$ and $T(1/t)$:
\begin{eqnarray*}
  0&=& t\, \frac{1-(tU)^{k+1}}{1-tU}- t\, \frac{(tU)^{k+1}}{1-tU} T(1/t)
- \frac{t^2\bU}{1-t\bU} T(t),
\\
0&=& t\, \frac{1-(t\bU )^{k+1}}{1-t\bU }- t\, \frac{(t\bU )^{k+1}}{1-t\bU } T(1/t)
- \frac{t^2U }{1-tU} T(t).
\end{eqnarray*}
 Solving this system gives in particular the value of
$T(1/t)$, and thus of $\PB^{(k)}(t)$ (thanks to~\eqref{C-W-h}). This provides the
second expression of  $\PB^{(k)}(t)$ given 
in Proposition~\ref{prop:culm-h}. 
The other results easily follow, using 
standard connections between
linear recurrence relations, their solutions,  and rational \gfs~\cite[Thm.~4.1.1]{stanley-vol1}.
\end{proof}

\subsection{\EE\SS\WW-bridges in the diagonal model}
%
\begin{Proposition}\label{prop:culm-d-1}
 Let $k\ge 0$. In  the diagonal model, the length \gf\ of
   \EE\SS\WW-pseudo-bridges of height $k$  is
$$
\PB_1^{(k)}(t) = \frac{t^k}{G_{k}(t)},
$$
where $G_k(t)$ is the   sequence of polynomials  defined by
$$
G_0=1, \quad G_1= 1-t^2\quad \hbox{and for } k\ge 1,
\quad G_{k+1}=(1+t^2)G_k -t^2(2-t^2)G_{k-1} .
$$
Equivalently,
$$
\sum_{k\ge 0} \frac{v^k t^k}{\PB_1^{(k)}(t) }= \sum_{k\ge 0}v^k G_k=
\frac{1-2t^2v}{1-(1+t^2)v+t^2(2-t^2)v^2},
$$
or
$$
\PB_1^{(k)}(t) = \frac{U-(2-t^2)\bU}
{(U-2t)U^k-\bigl((2-t^2)\bU-2t\bigr)\bigl((2-t^2)\bU\bigr)^k},
$$
where 
$$
U= {\frac {1+{t}^{2}
-\sqrt { \left(1- t^2 \right)   \left(1- 5\,{t}^{2} \right) }}
{2t }
}
$$
is a root of 
$t{u}^{2}- \left( 1+{t}^{2} \right) u+t \left( 2-{t}^{2} \right) 
=0$ 
and $\bU:=1/U$.
\end{Proposition}
\begin{proof}
The proof is very close to the proof of
Proposition~\ref{prop:culm-h},
but the role that was played by \EE\ steps is now played by
\SS\ steps. In particular, $\cT$ is now the set of \EE\SS\WW-walks
that end with a \SS\ step,  
  and in which each vertex $v$ satisfies $0\le h(v)\le k$. The sets  $\cT_i$
  and the series  $T_i(t)\equiv T_i$ and $T(t;u)\equiv T(u)$ are
  then defined in terms of $\cT$  as before.
Note that $T_k$ is in fact $0$.

 Pseudo-bridges of height $k$ containing
at least one \SS\ step are  obtained
by adding a sequence of  \EE\ steps of appropriate length to a walk of
$\cT$, and~\eqref{C-W-h} still holds (with $\PB^{(k)}$ replaced by $\PB_1^{(k)}$).
\begin{Lemma}\label{lem:culm-d-1}
  The series $T(t;u)$, denoted $T(u)$ for short,
 satisfies the following equation:
$$
\left(1-\frac{t^2}{1-tu}- \frac{t\bu}{1-t\bu}\right) T(u)
=
 t^2\, \frac{1-(tu)^{k}}{1-tu}-  t^2\, \frac{(tu)^{k}}{1-tu} T(1/t)
- \frac{t\bu}{1-t\bu} T(t),
$$
with $\bu=1/u$.
\end{Lemma}
\begin{proof}
We partition the set $\cT$ into three disjoint subsets. 
\begin{itemize}
\item 
The first subset  consists of walks with a single \SS\ step. Their \gf\ is  
$$
t\bu \sum_{i=1}^{k} (tu)^i= t^2\, \frac{1-(tu)^{k}}{1-tu}.
$$
 \item 
  The second subset consists of
walks in which the last \SS\ step is weakly higher than the
previous one. Their  \gf\  reads
$$
t\bu\, \sum_{i=0}^{k}\left(  T_i (t) u^i \sum_{j=1}^{k-i} (tu)^j \right)
=
\frac{t^2}{1-tu}T(u) - t^2\, \frac{(tu)^{k}}{1-tu} T(1/t).
$$
 \item 
  The third subset  consists of
walks in which the last \SS\  step is strictly lower than the
previous one. Their  \gf\  reads
$$
t\bu\, \sum_{i=0}^{k} \left(T_i (t) u^i \sum_{j=0}^{i-1} (t\bu)^j  \right)
=
\frac{t\bu}{1-t\bu}T(u)- \frac{t\bu}{1-t\bu} T(t).
$$
\end{itemize}
 Adding the three contributions establishes the lemma.
\end{proof}

Again, we solve the equation of Lemma~\ref{lem:culm-d-1} using the
kernel method, and conclude the proof of
Proposition~\ref{prop:culm-d-1} using~\eqref{C-W-h} (with $\PB^{(k)}$ replaced by $\PB_1^{(k)}$). 
\end{proof}

\subsection{\NN\EE\SS-bridges in the diagonal model}
The \gf\ of \NN\EE\SS-pseudo bridges is closely related to that of
\EE\SS\WW-pseudo-bridges. This will be explained combinatorially in
Section~\ref{sec:comments}, after a detour via partially directed \emm
excursions.,

\begin{Proposition}\label{prop:culm-d-2}
 Let $k\ge 0$. In  the diagonal model, the length \gf\ of
   \NN\EE\SS-pseudo-bridges of height $k$  is
$$
\PB_2^{(k)}(t) = \frac{t^k(2-t^2)^k}{G_{k}(t)},
$$
where $G_k(t)$ is the sequence of polynomials defined  in
Proposition~{\rm\ref{prop:culm-d-1}}.
In other words,
$$
\PB_2^{(k)}(t) = {(2-t^2)^k}\PB_1^{(k)}(t),
$$
so that other expressions of $\PB_2^{(k)}(t)$ can be derived from
Proposition~{\rm\ref{prop:culm-d-1}}.
%
%
\end{Proposition}
\begin{proof}
 The sets  $\cT$,  $\cT_i$, and the corresponding
series  $T_i(t)\equiv T_i$ and $T(t;u)\equiv T(u)$ are defined as in
the proof of  Proposition~\ref{prop:culm-h}
 --- only, the notion of height has changed. 
Note that $T_0$ is $0$.

 Pseudo-bridges of height $k$ containing
at least one \EE\ step are again obtained
by adding a sequence of  \NN\ steps of appropriate length to a walk of
$\cT$, and~\eqref{C-W-h} still holds (with $\PB^{(k)}$ replaced by $\PB_2^{(k)}$).
\begin{Lemma}\label{lem:culm-d-2}
  The series $T(t;u)$, denoted $T(u)$ for short,
 satisfies the following equation:
$$
\left(1-\frac{ut}{1-tu}- \frac{t^2}{1-t\bu}\right) T(u)
=
 tu\, \frac{1-(tu)^{k}}{1-tu}-  \frac{(tu)^{k+1}}{1-tu} T(1/t)
- \frac{t^2}{1-t\bu} T(t),
$$
with $\bu=1/u$.
\end{Lemma}
\begin{proof}
We partition the set $\cT$ into three disjoint subsets, defined as in
the proof of Lemma~\ref{lem:culm-h}. 
\begin{itemize}
\item 
The first subset  consists of walks with a single \EE\ step. Their \gf\ is  
$$
tu \sum_{i=0}^{k-1} (tu)^i= tu\, \frac{1-(tu)^{k}}{1-tu}.
$$
 \item  The second subset consists of
walks in which the last \EE\ step is strictly higher than the
previous one. Their  \gf\  reads
$$
tu\, \sum_{i=0}^{k}\left(  T_i (t) u^i \sum_{j=0}^{k-i-1} (tu)^j \right)
=
\frac{ut}{1-tu}T(u) -  \frac{(tu)^{k+1}}{1-tu} T(1/t).
$$
 \item The third subset  consists of
walks in which the last \EE\  step is weakly lower than the
previous one. Their  \gf\  reads
$$
tu\, \sum_{i=0}^{k} \left(T_i (t) u^i \sum_{j=1}^{i} (t\bu)^j  \right)
=
\frac{t^2}{1-t\bu}T(u)- \frac{t^2}{1-t\bu} T(t).
$$
\end{itemize}
 Adding the three contributions establishes the lemma.
\end{proof}

The kernel occurring in Lemma~\ref{lem:culm-d-2} is obtained by
replacing $u$ by $\bu$ in the kernel of
Lemma~\ref{lem:culm-d-1}. Consequently, it vanishes when
$u=1/U$ or $u=U/(2-t^2)$, where $U$ is the series defined in
Proposition~\ref{prop:culm-d-1}. 
We use the kernel method to solve
the equation of Lemma~\ref{lem:culm-d-2}, and conclude the proof of
Proposition~\ref{prop:culm-d-2} using~\eqref{C-W-h} (with $\PB^{(k)}$
replaced by $\PB_2^{(k)}$).   
\end{proof}

\subsection{\EE\SS-bridges in the diagonal model}
We state our last result on partially directed bridges  without proof, for two
reasons. Firstly, the 
step-by-step approach used in the previous subsections should have
become routine by now, and is especially simple to implement here. Secondly, 
this result already appears in~\cite[Prop.~3.1]{mbm-ponty} (where a bridge
preceded by an \EE\ step is called a \emm  culminating path,).

\begin{Proposition}\label{prop:culm-d-0}
 Let $k\ge 0$. In  the diagonal model, the length \gf\ of
   \EE\SS-pseudo-bridges of height $k$  is
$$
\PB_0^{(k)}(t) = \frac{t^k}{F_{k}(t)},
$$
where $F_k(t)$ is the sequence of polynomials defined by
$$
F_{-1}=1, \quad F_0=1, 
\quad \hbox{and for } k\ge 0,\quad 
F_{k+1}=F_k -t^2F_{k-1} .
$$
Equivalently,
$$
\sum_{k\ge 0} \frac{v^k t^k}{\PB_0^{(k)}(t) }= \sum_{k\ge 0}v^k F_k=
{\frac {1-v{t}^{2}}{1-v+{v}^{2}{t}^{2}}}
,
$$
or
$$
\PB_0^{(k)}(t) = \frac{U^2-\bU^2}{U^{k+2}-\bU^{k+2}},
$$
where 
$$
U=\frac{1-\sqrt{1-4t^2}}{2t}
$$
is a root of $ tu^2-u+t=0$ 
and $\bU:=1/U$ is the other root of this polynomial.
\end{Proposition}

\section{Partially directed bridges via heaps of cycles}
\label{sec:culminant-heaps}

In this section, we give alternative (and  more combinatorial)
proofs of the results of
Section~\ref{sec:culminant-kernel}. In particular, these proofs
explain why the numerators of the rational series 
that count partially directed
bridges of height $k$ are so simple ($t^k$ or $t^k(2-t^2)^k$,
depending on the model).

 As a preliminary observation, let
us note that  \EE\SS-pseudo-bridges of height $k$ in the diagonal model can be
seen as arbitrary paths on the segment $\{0,1, \ldots, k\}$, with 
steps $\pm 1$, going from $0$ to $k$. Therefore, a natural way to count them
is to use a classical
 result that expresses the \gf\ of paths with
prescribed endpoints in a directed graph.
 This result is recalled
 in Proposition~\ref{prop:inversion} below. 
It gives a straightforward proof of
Proposition~\ref{prop:culm-d-0}.
However, the other three classes of bridges that we have
counted do not fall
immediately in the scope of this general result, because of the
self-avoidance condition (which holds automatically for
\EE\SS-walks). For instance, in the horizontal model, a 
\NN\EE\SS-pseudo-bridge of height $k$ is not an \emm arbitrary, path with steps $0,
\pm 1$ going from $0$ to $k$ on the  segment $\{0,1, \ldots, k\}$. 
 We show here how to recover the results of
Section~\ref{sec:culminant-kernel} by factoring bridges into
more general steps, and then applying  Proposition~\ref{prop:inversion}.

Let $\Gamma=(V,E)$ be a (finite) directed graph. To each 
arc  of this graph, we
associate a {weight}  taken in some commutative ring (typically,
a ring of formal power series). 
A \emph{cycle} of $\Gamma$ is a path ending at its starting point, taken up to
a cyclic permutation. A path  is \emph{self-avoiding} if it does not
visit the same vertex twice. A (non-empty) self-avoiding cycle is
called an \emph{elementary cycle}. Two paths are \emph{disjoint} if
their vertex sets are disjoint. The \emm weight,  $w(\pi)$ of  a path (or cycle)
$\pi$  is the product of the weights of its arcs.  
A \emph{configuration of cycles}
$\ga=\{\ga_1,\dotsc,\ga_r\}$ is a set of  pairwise disjoint elementary cycles.
The \emph{signed weight}  of $\ga$ is
$$\widetilde w(\ga):= (-1)^r \prod_{i=1}^r w(\ga_i).$$
For  two vertices  $i$ and $j$,  denote by $W_{i,j}$ the \gf\ of  paths going from from $i$ to $j$:
\[W_{i,j}=\sum_{\pi: i \leadsto j} w(\pi).\]
We
assume that this sum is well-defined, which is  always the case when
$W_{i,j}$ is a length \gf.

\begin{Proposition}%
\label{prop:inversion}
The \gf\ of paths going from $i$ to $j$ in the weighted digraph
$\Gamma$ is
\[W_{i,j}=\frac{N_{i,j}}{G}\text,\]
where $G=\sum_\ga\wt(\ga)$  
is the signed \gf \ of  configuration of cycles, and 
\[N_{i,j}=\sum_{\eta,\ga}w(\eta)\wt(\ga)\text,\]
where $\eta$ is a self-avoiding path going from $i$ to $j$ and $\ga$ is a
configuration of cycles disjoint from $\eta$.
\end{Proposition}
This classical result can be proved as follows: one first identifies 
$W_{i,j}$ as the $(i,j)$ coefficient of the matrix $(1-A)^{-1}$, where $A$ is the
weighted adjacency matrix of $\Gamma$. Thanks to standard linear
algebra, this coefficient can be expressed in
terms of the determinant  of $(1-A)$ and one of its
cofactors~\cite[Thms.~4.7.1 and 4.7.2]{stanley-vol1}. 
A simple expansion
of these as sums over permutations shows that the determinant is $G$,
and the cofactor $N_{i,j}$. 
Proposition~\ref{prop:inversion}  can also be proved without
any reference to linear algebra, using the theory of \emm partially
commutative monoids,, or, more geometrically, \emm heaps of pieces,~\cite{foata,viennot}.
 In this context, configurations of cycles are called
\emm trivial heaps of cycles,. This is the only justification of the
title of this section, where no non-trivial heap will actually be seen.

\medskip
As a straightforward application, let us sketch a second proof of
Proposition~\ref{prop:culm-d-0}. The vertices of $\Gamma$ are $0,1, \ldots,
k$, with an arc from $i$ to $j$ if $|i-j|=1$. 
We apply the above
proposition to count paths going from $0$ to $k$. 
All arc weights are $t$.
The elementary cycles
have length 2, and by induction on $k$, it is easy to see that $G$, the
signed \gf \ of  configurations of cycles, is the polynomial $F_k$. The
only self-avoiding 
path $\eta$ going from $0$ to $k$ consists of $k$ `up' steps and visits
all vertices, so that $N_{0,k}=t^k$. Proposition~\ref{prop:culm-d-0} follows.

\subsection{Bridges with large down steps}

The  proof of
Proposition~\ref{prop:culm-d-0} that we have just sketched can be
 extended to paths with 
arbitrary large down steps. This will be used below to count partially
directed bridges.

Let $\Gamma_k$ be the graph
with vertices $\{0,\dotsc,k\}$ and with the following weighted arcs:
\begin{itemize}
\item \emph{ascending} arcs  $i\to i+1$ of height $1$, with weight~$A$, for $i=0, \ldots,k-1$;
\item \emph{descending} arcs  $i\to i-h$ of height $h$, with weight~$D_h$, for
  $i=h,\ldots,k$ and  $h\geq0$.
\end{itemize}
For $k\geq0$, denote by $C^{(k)}$ the \gf\ of paths from $0$ to $k$ in
the graph $\Gamma_k$. These paths may be seen as pseudo-bridges
  of height $k$ with general down steps.

\begin{Lemma}\label{lem:culminating}
The \gf\ of pseudo-bridges of height $k$ is
\[C^{(k)}=\frac{A^k}{H_k},\]
where the \gf\ of the denominators $H_k$ is
\beq
\sum_{k\geq0}H_kv^k=\frac{1-D(vA)}{1-v+vD(vA)},
\label{H-ser}
\eeq
with $D(v)$ the \gf\ of descending steps:
\[D(v)=\sum_{h\geq0}D_hv^h.\]
\end{Lemma}
\begin{proof}
With the notation of  Proposition~\ref{prop:inversion}, the series
$C^{(k)}$ reads $N_{0,k}/G$.
Since  all ascending arcs have height~1, the only
self-avoiding path from $0$ to $k$ consists of $k$ ascending arcs, and has
weight $A^k$. As it visits every vertex of the graph, the only configuration
of cycles disjoint from it is the empty configuration. Therefore, the 
numerator $N_{0,k}$ is simply $A^k$.
The   elementary cycles consist of a
descending step of height, say, $h$, followed by $h$ ascending steps. The weight of this cycle  is $D_hA^h$.

To underline the dependence of our graph in $k$, denote by $H_k$ the
denominator $G$ of Proposition~\ref{prop:inversion}.
Consider a configuration of cycles of $\Gamma_k$: either the
vertex $k$ is free, or it is occupied by a cycle; this gives the
following recurrence relation, valid for $k\ge 0$:
$$
H_k=H_{k-1}-\sum_{h=0}^kD_hA^hH_{k-h-1},
$$
with the initial condition $H_{-1}=1$. This is equivalent to~\eqref{H-ser}.
\end{proof}

\subsection{Partially directed self-avoiding walks as arbitrary paths}

As discussed above, it is not straightforward to apply Proposition~\ref{prop:inversion} (or
Lemma~\ref{lem:culminating}) to the enumeration  of  partially
directed bridges, because of the self-avoidance condition. To circumvent
this difficulty, we will 
first prove that  partially directed self-avoiding walks are arbitrary paths on a
line with large down steps.

It will be convenient to  regard lattice walks as words on the alphabet
$\{\NN,\EE,\SS,\WW\}$, and  sets of walks as 
languages. We thus use some standard notation from the  theory of
formal languages~\cite{hopcroft}. 
 The length of a word $u$ (the number of letters) is
denoted by $|u|$, and the number of occurrences of the letter $a$ in
$u$ is
$|u|_a$. For  two languages $\cL$ and $\cL'$,  
\begin{itemize}
\item $\cL+\cL'$ denotes the union of $\cL$ and $\cL'$;
\item $\cL\cL'$ denotes the language formed of all concatenations of a word of
$\cL$ with a word of $\cL'$;
\item $\cL^*$ denotes the language formed of all sequences of words of $\cL$;
\item $\cL^+$ denotes the language formed of all \emph{nonempty} sequences of
words of $\cL$.
\end{itemize}
Finally, for any letter $a$, we denote by $a$ the \emm elementary, language $\{a\}$.
A \emm  regular expression, is any expression obtained from 
elementary languages using 
the sum, product, star and plus operators. 
It is \emph{unambiguous} if every word of the corresponding language
has a unique factorization compatible with the expression.
To take a simple example, the expressions $(\NN+\EE)^*$ and $(\NN+\WW)^*$ are
unambiguous expressions describing \NN\EE- and \NN\WW-walks respectively.
However, the expression $(\NN+\EE)^* + (\NN+\WW)^*$ is ambiguous, as every
\NN-walk is matched twice.
Unambiguous regular expressions
translate directly into enumerative results.

\medskip
Let us say that a \NN\EE\SS-walk is 
\emph{proper} if it neither begins nor ends with a \SS\ step.  All
 \NN\EE\SS-pseudo-bridges are proper, whether in the horizontal or
diagonal model. The following lemma explains how to see  proper
\NN\EE\SS-walks  as sequences of generalized steps.

\begin{Lemma}\label{lem:proper}
Every proper \NN\EE\SS-walk has a unique factorization into \NN\ steps and
nonempty proper \EE\SS-walks with no consecutive \EE\ steps. In other
words, the language 
of proper \NN\EE\SS-walks admits 
the following unambiguous regular expression:
\[
\bigl(\NN+\EE(\SS^+\EE)^*\bigr)^*.\]
\end{Lemma}

\begin{proof}  The factorization of proper \NN\EE\SS-walks  is
  exemplified in Fig.~\ref{fig:proper}. Every \NN\ step is a factor,
  as well as every maximal   \EE\SS-walk with no consecutive \EE\ steps.
\end{proof}

\begin{figure}[ht]
\centering
\includegraphics[scale=0.6]{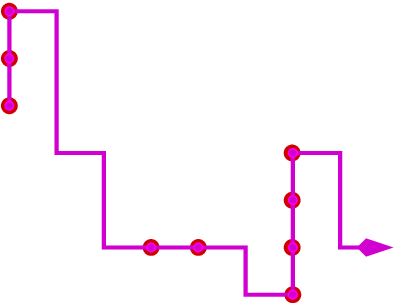}
\caption{
The factorization of a proper \NN\EE\SS-walk.}
\label{fig:proper}
\end{figure}

A similar result holds for  \EE\SS\WW-walks (which we need to study in
the diagonal model), which are obtained by applying a quarter turn to
\NN\EE\SS-walks.  Let us say that  an
\EE\SS\WW-walk is  \emph{proper} if it neither begins nor ends with a
\WW\ step.
After a rotation,  Lemma~\ref{lem:proper} gives for the language of
proper \EE\SS\WW-walks  the following unambiguous description:
\beq\label{fact-ESW}
\bigl(\EE+\SS(\WW^+\SS)^*\bigr)^*.
\eeq

\subsection{Partially directed bridges}

We can now give new proofs of the results of
Section~\ref{sec:culminant-kernel},
based on Lemma~\ref{lem:culminating}.

\medskip
\noindent 
{\em Second proof of Proposition}~\ref{prop:culm-h}.
Thanks to  Lemma~\ref{lem:proper},  self-avoiding
\NN\EE\SS-pseudo-bridges of height $k$ can be
seen as arbitrary pseudo-bridges  of height $k$ (in the sense of
Lemma~\ref{lem:culminating}) where \NN\ is the only 
ascending step (of height 1 and weight $t$), and all words of  $\EE(\SS^+\EE)^*$
are descending 
steps. Moreover, the weight of  a descending step $u$ is $t^{|u|}$ and its
height is $|u|_\SS$. Thus, with the notation of
Lemma~\ref{lem:culminating}, $A=t$ and the \gf\  $D(v)$ of descending
steps is derived  from the
regular expression  $\EE(\SS^+\EE)^*$:
\begin{align*}
D(v)&=\frac{t}{1-\frac{t^2v}{1-tv}}.
\end{align*}
Proposition~\ref{prop:culm-h}, in the form~\eqref{pb-inv}, 
now follows from Lemma~\ref{lem:culminating}.
\qed

\medskip
\noindent 
{\em Second proof of Proposition}~\ref{prop:culm-d-1}.
Thanks to~\eqref{fact-ESW}, self-avoiding \EE\SS\WW-pseudo-bridges of height $k$  can be
seen as arbitrary pseudo-bridges  of height $k$  (in the sense of
Lemma~\ref{lem:culminating}) where \EE\ is the only 
ascending step (of  weight $t$), and all words of  $\SS(\WW^+\SS)^*$
are descending  
steps. Moreover, the weight of a descending step $u$ is $t^{|u|}$
and its height is  $|u|$. Thus, with the notation of 
Lemma~\ref{lem:culminating}, $A=t$ and the \gf\  $D(v)$ of descending
steps is 
$$
D(v)=\frac{tv}{1-\frac{t^2v^2}{1-tv}}.$$
Proposition~\ref{prop:culm-d-1} now follows from 
Lemma~\ref{lem:culminating}.
\qed

\medskip
\noindent 
{\em Second proof of Proposition}~\ref{prop:culm-d-2}.
Again,  the description of  \NN\EE\SS-walks  given by
Lemma~\ref{lem:proper} allows us to regard these self-avoiding walks
as arbitrary paths with generalized steps.  
In  the diagonal framework, the ascending steps $u$ are  $\NN$ and
all words of 
$\EE(\SS\EE)^*$. They all have weight $t^{|u|}$.  All
other words of $\EE(\SS^+\EE)^*$ are descending.
 Moreover, the weight of a descending step $u$ is  $t^{|u|}$ and its
 height is  $|u|_\SS-|u|_\EE$.
 Thus, with the notation of Lemma~\ref{lem:culminating}, 
$$
A=t+\frac{t}{1-t^2}=\frac{t(2-t^2)}{1-t^2}\quad \hbox{and} \quad
D(v)=\frac{tv^{-1}}{1-\frac{t^2}{1-tv}}-\frac{tv^{-1}}{1-t^2}.
$$

However, one must pay attention to the following detail: in a
  \NN\EE\SS-pseudo-bridge of height $k$, only the last 
generalized step can end at height $k$, because all descending steps
begin with \EE. Similarly, all descending steps end with \EE, which
implies that the only  generalized step that starts at height
$0$ is
the first one (and moreover it is an ascending step). Thus a
  \NN\EE\SS-pseudo-bridge of height $k\ge 2$ is 
really a pseudo-bridge (in the sense of Lemma~\ref{lem:culminating})   \emm of
height $k-2$,, preceded and followed by an ascending
step. Thus for $k\ge 2$,
$$
\PB_2^{(k)}= \frac{A^k}{H_{k-2}},
$$
where the \gf\ of the denominators $H_k$ is given in Lemma~\ref{lem:culminating}. 
Given that $\PB_0^{(2)}=1$ and  $\PB_1^{(2)}=A$, we have
$$
\sum_{k\ge 0} \frac{t^k(2-t^2)^k v^k}{B^{(k)}_2(t)}=
1+ \frac{t(2-t^2) v}{A} + \sum_{k\ge 2} \frac{t^k(2-t^2)^k
  v^k}{A^k}H_{k-2}
=\frac{1-2t^2v}{1-(1+t^2)v+t^2(2-t^2)v^2}.
$$
This is equivalent to   Proposition~\ref{prop:culm-d-2}.
\qed

\section{Weakly directed  walks: the horizontal model}
\label{sec:gf}

We now return to the weakly directed walks defined in
Section~\ref{sec:def}.  We determine their  \gf, study their
asymptotic number and average end-to-end distance, and finally prove
that the \gf\ we have obtained has infinitely many singularities, and
hence, cannot be D-finite.

\subsection{Generating functions}
\label{sec:gf1}

By combining Propositions~\ref{prop:equiv} and~\ref{prop:culm-h}, it
is now easy to count  weakly directed \emm  bridges,.

\begin{Proposition}\label{prop:weakly-h}
In the horizontal model, the \gf\ of  weakly directed bridges is:
$$
W(t)= 
\frac 1 {1+t-\frac{2t\PB}{1+t\PB}}
$$
where 
$\PB
:=\sum_{k\ge0}\PB^{(k)}(t) $ is the \gf\ of
\NN\EE\SS-pseudo-bridges, given by~Proposition~{\rm\ref{prop:culm-h}}.
\end{Proposition}
\begin{proof}
  Let $\cI_\EE$ be the set of irreducible \NN\EE\SS-bridges, and let
  $I_\EE(t)$ be the associated length \gf. We will most of the time omit the variable
  $t$ in our series, writing for instance $I_\EE$ instead of
  $I_\EE(t)$.
 Given that a non-empty
\NN\EE\SS-bridge is obtained by adding a \NN\ step at
 the end of a \NN\EE\SS-pseudo-bridge, and is
  a (non-empty) sequence of irreducible \NN\EE\SS-bridges, we have:
$$
t \PB= \frac {I_\EE}{1-I_\EE}.
$$
Define  similarly the set $\cI_\WW$, and the associated series 
$I_\WW$. By symmetry, $I_\WW=I_\EE$.
 Moreover,
$$
\cI_\EE\cap \cI_\WW=\NN.
$$
Hence the \gf\ of  irreducible bridges that are either \NN\EE\SS\ or
\NN\SS\WW\ is 
$$
I:=I_\EE+ I_\WW - t= \frac{2t \PB}{1+t\PB}-t.
$$
By Proposition~\ref{prop:equiv}, the \gf\ of  weakly directed bridges
is $ W=\frac 1{1-I}$. The result follows.
\end{proof}

We will now  determine the \gf\ of (general) weakly directed \emm walks.,
As we did for bridges, we factor them into ``irreducible'' factors, but
the first and last 
factors are not necessarily bridges, so that we need to extend the notion
of irreducibility to more general walks. Let us say that
a walk $v_0\dotsm v_n$ is
\emph{positive} if all its vertices $v$ satisfy $h(v)\geq h(v_0)$, and
that it is
\emph{copositive} if all vertices $v\neq v_n$ satisfy $h(v)<h(v_n)$.
Thus a bridge is a positive and copositive walk.

\begin{Definition}\label{def:irreducible}
Let $r$ denote the reflection through the $x$-axis.
A non-empty walk $w$ is \emph{\NN-reducible} if it is of the form $qp$, where $q$
is a nonempty copositive walk and $p$ is a nonempty positive walk. It is
\emph{\SS-reducible} if $r(w)$ is \NN-reducible. Finally, it is
\emph{irreducible} if it is neither \NN-reducible nor \SS-reducible.
\end{Definition}

We can rephrase this definition as follows. If a horizontal line at
height $h+1/2$, with $h \in \zs$, meets $w$ at exactly one point, we
say that the step of $w$ containing this point is a \emm separating,
step. Of course, this step is either \NN\ or \SS. Then a non-empty walk is
irreducible if it does not contain any non-final separating step.
It is then clear that the above definition extends the notion of irreducible bridges
defined in Section~\ref{sec:def}: a non-empty bridge is never \SS-reducible, and
it is \NN-reducible if and only if it is the product of two non-empty bridges.
Also, observe that the endpoint of a \NN-reducible walk is strictly
higher than its origin: Thus a walk may not be both \NN-reducible and
\SS-reducible.

By cutting a walk after each separating step, one obtains  a 
decomposition into a sequence of irreducible walks. 
This may be either a \NN-decomposition or a \SS-decomposition. 
The first factor  of a \NN-decomposition is copositive, while the
last one is positive. The intermediate factors are bridges.

We can now generalize Proposition~\ref{prop:equiv},
and characterize weakly directed walks in terms their irreducible factors.

\begin{Proposition}\label{prop-equiv-general}
A walk is weakly directed if and only if each of its irreducible factors
is partially directed.
Equivalently, each of these factors is a \NN\EE\SS- or a \NN\SS\WW-walk.
\end{Proposition}

\begin{proof}
The proof is very similar to that of Proposition~\ref{prop:equiv}.
First, the equivalence between the two conditions comes from the fact
that any partially directed irreducible walk is  \NN\EE\SS\ or
\NN\SS\WW.

Now, if all
irreducible factors of a walk are partially directed, then this walk is weakly
directed: two points of the walk lying on the same horizontal line
belong to the same irreducible factor.

Conversely, let $w$ be an irreducible factor of a weakly directed walk; then
$w$ is weakly directed. We prove that it is either a \NN\EE\SS- or a
\NN\SS\WW-walk.
Assume that this is not the case, i.e.\ $w$ contains both a \WW\  and an \EE\ 
step. By symmetry, we may assume that it contains a \WW\  step
before its first
\EE\  step. By symmetry again,  we may assume that, between  the first
\EE\  step and the last \WW\ step that precedes it, the walk consists
of \NN\ steps. Then
the first argument used in the proof of Proposition~\ref{prop:equiv},
depicted in the first part of
Fig.~\ref{fig:preuve-equiv},  leads to a contradiction.
\end{proof}

We now proceed to the enumeration of general weakly directed walks.
\begin{Lemma}\label{lem:TPQ}
  The generating functions $T(t)$, $P(t)$ and $Q(t)$ of general,
  positive and copositive \NN\EE\SS-walks are:
\begin{align*}
T(t)&=\frac {1+t}{1-2t-t^2},\\
P(t)&=\frac1{2t^2}\Biggl(\sqrt{\frac{1-t^4}{1-2t-t^2}}-1-t\Biggr),\\
Q(t)&=1+tP(t).
\end{align*}
\end{Lemma}

\begin{proof}
Let us start with general \NN\EE\SS-walks. The language $\cT$ of these
walks is given by  the following unambiguous description:
$$
\cT= \NN^*+\SS^++ \cT\EE(\NN^*+\SS^+),
$$
from which the expression of $T(t)$ readily follows.

Let us now count positive walks. Let $P(t;u)$
be their generating function, with the variable $u$ accounting for the height of
the endpoint. 
We decompose positive walks by cutting them before the last
\EE\  step; this is similar to what we did in the proof   of
Lemma~\ref{lem:culm-h}. We thus obtain:
\[P(t;u)=\frac1{1-tu}+\frac{t^2uP(t;u)}{1-tu}
+\frac{t\bigl(P(t;u)-t\bu P(t;t)\bigr)}{1-t\bar u}.\]
We rewrite this as follows:
\[\biggl(1-\frac{t^2u}{1-tu}-\frac{t}{1-t\bar u}\biggr)P(t;u)
=\frac1{1-tu}-\frac{t^2\bar uP(t;t)}{1-t\bar u}.\]
We apply  again the kernel method: we specialize $u$ to the 
series $U$ of Proposition~\ref{prop:culm-h}; this cancels the
coefficient of $P(t;u)$, and we thus obtain  the value of $P(t;t)$. We
then specialize the above equation to $u=1$ to determine $P(t;1)$, which is 
the series denoted $P(t)$ in the lemma.

Finally, a non-empty copositive walk is obtained by reading a positive walk, 
seen as a word on $\{\NN,\EE, \SS, \WW\}$, 
from right to left, and adding a final \NN\ step. This gives the last
equation of the lemma. 
\end{proof}

\begin{Proposition}\label{prop:general-weakly-h}
The generating function  of weakly directed walks is
\[
\overline W(t)=
1+\bigl(2T_i(t)-2t\bigr)+2\bigl(2Q_i(t)-t\bigr)W(t)\bigl(2P_i(t)-t\bigr),
\]
where the series $T_i(t)$, $P_i(t)$  and $Q_i(t)$  count respectively
 general,
positive, and copositive irreducible \NN\EE\SS-walks, and are given by:
\begin{align*}
T_i(t)&=T(t)-1-2Q_i(t)\bigl(1+tB(t)\bigr)P_i(t),
&P_i(t)&=\frac{P(t)-1}{1+tB(t)},&Q_i(t)&=\frac{Q(t)-1}{1+tB(t)}.&
\end{align*}
The series $W$, $\PB$, $T$, $P$ and $Q$ are those of
Proposition~{\rm\ref{prop:weakly-h}} and Lemma~\rm{\ref{lem:TPQ}. } 
\end{Proposition}

\begin{proof}
In order to determine the series $T_i(t)$, $P_i(t)$  and $Q_i(t)$, we
decompose into irreducible factors the corresponding   families of
\NN\EE\SS-walks. 
\begin{itemize}
\item A general \NN\EE\SS-walk is either
\begin{itemize}
\item empty, or
\item irreducible, or
\item \NN-reducible: in this case, it consists of an irreducible copositive
\NN\EE\SS-walk, followed by a sequence of irreducible
\NN\EE\SS-bridges  (forming a 
\NN\EE\SS-bridge), and finally by an irreducible positive \NN\EE\SS-walk; or 
\item  symmetrically, \SS-reducible.
\end{itemize}
Since $1+tB(t)$ is the \gf\ of bridges, this gives
\[T(t)=1+T_i(t)+2Q_i(t)\bigl(1+tB(t)\bigr)P_i(t).\]
\item We now specialize the above decomposition to  positive
  \NN\EE\SS-walks. Observe that when such a walk is \NN-reducible, its
  first factor is a bridge. This allows us to merge the second and
  third cases above. Moreover, the
  fourth case never occurs. Thus a  positive   \NN\EE\SS-walk is either
\begin{itemize}
\item empty, or
\item a \NN\EE\SS-bridge followed by an irreducible positive \NN\EE\SS-walk.
\end{itemize}
This yields:
\[P(t)=1+\bigl(1+tB(t)\bigr)P_i(t).\]
\item We proceed similarly for  copositive \NN\EE\SS-walks. Such a
  walk is either 
\begin{itemize}
\item empty, or
\item an irreducible   copositive \NN\EE\SS-walk followed by a \NN\EE\SS-bridge. 
\end{itemize}
This gives:
\[Q(t)=1+Q_i(t)\bigl(1+tB(t)\bigr).\]
\end{itemize}
We thus obtain the expressions of $T_i$, $P_i$ and $Q_i$ announced in the
proposition.

\medskip
Recall from Proposition~\ref{prop-equiv-general} that a walk is weakly
directed if and only if its irreducible factors
are \NN\EE\SS- or \NN\SS\WW-walks.
Thus,
\begin{itemize}
\item a weakly directed walk is either
\begin{itemize}
\item empty, or
\item an irreducible \NN\EE\SS- or \NN\SS\WW-walk, or
\item \NN-reducible:
it then factors into an irreducible copositive \NN\EE\SS- or \NN\SS\WW-walk,
a sequence of \NN\EE\SS- or \NN\SS\WW- irreducible bridges (forming a
weakly directed bridge), and an irreducible positive \NN\EE\SS- or
\NN\SS\WW-walk; or 
\item symmetrically, \SS-reducible.
\end{itemize}
\end{itemize}
The contribution to $\overline W(t)$ of the first case is obviously 1.
The only irreducible walks that are both \NN\EE\SS\  and
\NN\SS\WW\  are \NN\  and 
\SS.
The generating function of 
 irreducible \NN\EE\SS- or \NN\SS\WW-walks is thus $2T_i(t)-2t$.
Similarly, the \gf\
of irreducible positive (resp. copositive) \NN\EE\SS- or \NN\SS\WW-walks
is $2P_i(t)-t$ 
(resp. $2Q_i(t)-t$). In both cases, the term $-t$ corresponds to the walk 
reduced to a \NN\ step, which is both \NN\EE\SS\ and \NN\SS\WW.
Adding the contributions of the four classes yields the announced
expression of  $\overline W(t)$.
\end{proof}

\subsection{Asymptotic results}
\label{sec:asympt}
\begin{Proposition}\label{prop:asympt-h}
 The \gf\ $W$ of  weakly directed
  bridges, given in Proposition~{\rm\ref{prop:weakly-h}}, is
  meromorphic in the disk 
  $\cD=\{z: |z|< \sqrt 2 -1\}$. It has a unique dominant pole in this disk,
  $\rho\simeq 0.3929$. This pole is simple. Consequently,  the number
  $w_n$ of   weakly  directed bridges of length $n$ satisfies
\[
w_n \sim \kappa \mu^n,
\]
with $\mu=1/\rho\simeq 2.5447$. 

Let $N_n$ denote the number  of irreducible factors in a random
 weakly directed bridge of length $n$. The mean and
variance of $N_n$ satisfy:
\[
\E(N_n) \sim  \mathfrak m  \, n,  \quad  \quad
\Var(N_n) \sim  \mathfrak s ^2 \, n ,  
\]
where 
\[
 \mathfrak m  \simeq 0.318
\quad \hbox{and} \quad
\mathfrak s ^2 \simeq 0.7,
\]
 and the random variable
$
\frac{N_n-  \mathfrak m \, n}{\mathfrak s \sqrt n}
$
converges in law to  a standard normal distribution. In particular,
the average end-to-end distance, being bounded from below by
$\E(N_n)$, grows linearly with~$n$.

These results hold as well for general weakly directed walks, with
other values of $\kappa$, $\mathfrak m $ and~$\mathfrak s$.
\end{Proposition}
\begin{proof}
%
Recall from the proof of Proposition~\ref{prop:weakly-h} that
$W(t)=1/(1-I(t))$, 
  where $I(t)$ counts partially directed  irreducible bridges, which
  are certain \NN\EE\SS- or \NN\SS\WW-walks.  The
  \gf\ $T(t)$ of  \NN\EE\SS-walks, given  in 
 Lemma~\ref{lem:TPQ}, has radius of convergence $\sqrt 2 -1$. Hence,
 $I$ has radius of convergence at least  $\sqrt 2 -1$, and $W$
 is meromorphic in the disk $\cD$.  

In this disk, we find a pole at each value of $t$ for which $I(t)=1$.
As $I(t)$ has non-negative coefficients
and is aperiodic,
  a pole of minimal modulus,
if it exists, can only be real, positive and simple. Thus if there is a pole
in $\cD$, then $W$ has a unique \emm dominant, pole $\rho$, which is
simple, and the
asymptotic behaviour of the numbers $w_n$ follows. 

In order to prove the existence of $\rho$, we use upper and lower
bounds on the series $I(t)$. For any series $F(t)=\sum_{m\ge 0} f_m
t^m$, and  $n\ge 0$, denote $F_{\le n}(t):= \sum_{m= 0}^n f_m
t^m$ and  $F_{> n}(t) :=\sum_{m>n} f_m t^m$.
Then for $0<t<\sqrt 2 -1$ and  $n\ge 0$, we have
\beq\label{A-bound-h}
I^-(t) \le I(t) \le I^+(t),
\eeq
where the series
$$
I^-(t):= I_{\le n}(t)
\quad \hbox{ and } \quad 
I^+(t):= I_{\le n}(t) + 2T_{>n}(t)=  I_{\le n}(t) + 2T(t)- 2T_{\le n}(t)
$$
can be evaluated exactly for a given value of $n$.
The upper bound follows from the fact that $I$ counts walks that are
either \NN\EE\SS- or \NN\SS\WW-walks. Using  these bounds,  we can  prove the existence of $\rho$ and
locate it. More precisely, 
\beq\label{rho-bound-h}
\rho^-\le \rho \le \rho^+,
\eeq
 where
$$
I^-(\rho^+) = I^+(\rho^-)=1.$$
Taking $n=300$ gives 5 exact digits in $\mu=1/\rho$.

\medskip
Let us now study the number of irreducible bridges
in a random weakly directed bridge.
  The series that counts these
 bridges by their length and the number of irreducible bridges is
\beq\label{Wtx}
W(t,x)= \frac 1 {1-x I(t)}.
\eeq
One easily checks that $W(t,x)$ corresponds to a  \emm supercritical
sequence,, so that Prop.~IX.7  of~\cite{flajolet-sedgewick} applies and
establishes the existence of a gaussian limit law, after
standardization. Regarding the estimates of $\mathfrak m$
 and $\mathfrak s$,
we have 
$$ 
\mathfrak m= \frac 1{\rho I'(\rho)} \quad \hbox{and} \quad 
\mathfrak s^2= \frac {I''(\rho)+I'(\rho)-I'(\rho)^2}
{\rho I'(\rho)^3} 
.$$
As $I(t)$ has non-negative coefficients, we can combine the
bounds~\eqref{A-bound-h} on $I(t)$ 
and~\eqref{rho-bound-h} on $\rho$ to obtain bounds on the values of
$\mathfrak m$ and $\mathfrak s$.

\medskip Consider now the \gf\ $\overline W$ of general weakly
directed walks, given in Proposition~\ref{prop:general-weakly-h}. The
series $T_i$, $P_i$ and $Q_i$ count certain partially directed walks,
and thus have radius at least $\sqrt 2 -1$. Moreover, $2Q_i(t)-t>0$
and $2P_i(t)-t>0$ for $t>0$. Hence $\overline W$ has, as $W$ itself,  a unique dominant
pole in $\cD$, which is $\rho$. 

 The argument used to prove Proposition~\ref{prop:general-weakly-h}
 shows that the series that 
counts weakly directed walks by their length and the number of
irreducible factors  is
$$
\overline W(t,x)= 1+ x(2T_i(t)-2t) + 2x^2(2Q_i(t)-t)W(t,x)(2P_i(t)-t),
$$
where $W(t,x)$ is given by~\eqref{Wtx}. This yields the announced
results on the number of irreducible factors. 
\end{proof}

\subsection{ Nature of the series}
\label{sec:nature}
\begin{Proposition}\label{prop:nature}
  The \gf\ $\PB=\sum_{k\ge 0} B^{(k)}(t)$ of \NN\EE\SS-pseudo-bridges, given in
  Proposition~{\rm\ref{prop:culm-h}},  converges   around $0$ and has
a  meromorphic
continuation  in $\cs\setminus \cE$, where $\cE$ consists of the
two real intervals $[-\sqrt 2-1, -1]$ and $[\sqrt 2-1,1]$, and of the curve
$$
\cE_0=\left\{x+iy: x\ge 0, \ {y}^{2}={\frac {1-{x}^{2}-2\,{x}^{3}}{1+2\,x}}
\right\}.
$$
This curve, shown in Fig.~{\rm\ref{fig:crit}}, is a natural boundary of
$\PB$. That is, every point of $\cE_0$ is a singularity of $\PB$.

The above statements  hold as well for the \gf\  $W$ of weakly
directed bridges, given in Proposition~{\rm\ref{prop:weakly-h}}. In
particular,  neither $\PB$ nor $W$ is D-finite. 
\end{Proposition}

\begin{figure}[htb]
\includegraphics[scale=0.4]{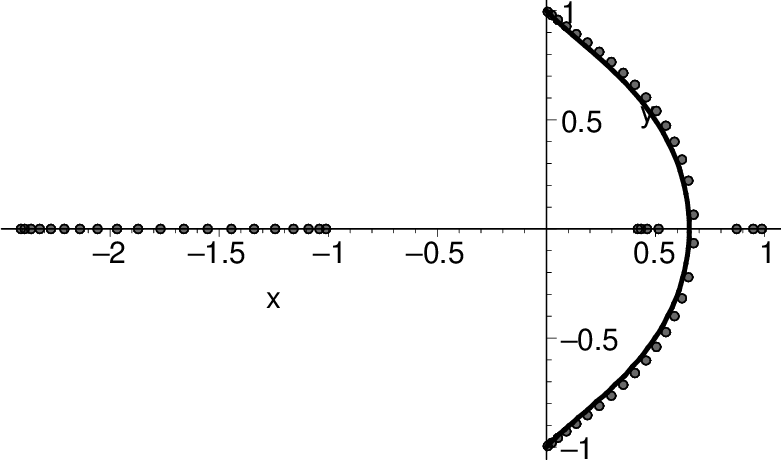}
\caption{The curve $\cE_0$ and the zeroes of $G_{20}$.} 
\label{fig:crit}
\end{figure}

Before proving this proposition, let us establish two lemmas, dealing
respectively with the series $U$ and the polynomials $G_k$ occurring
in the expression of $\PB$ (Proposition~\ref{prop:culm-h}). 

\begin{Lemma}
\label{lem:U}
For $t\in \cs\setminus\{0\}$, the equation
$t(u+1/u)=1-t+t^2+t^3$ has two roots, counted with multiplicity. The
product of these roots is $1$. Their modulus is $1$ if and only if $t$
belongs to the set $\cE$ defined in Proposition~{\rm\ref{prop:nature}}.

Let 
$$
U(t)= {\frac {1-t+{t}^{2}+{t}^{3}-
\sqrt { \left(1- t^4 \right)   \left(1-2t- {t}^{2}\right)  }}
{2t}}
$$
 be the root that is defined at $t=0$.
   This series has radius of convergence $\sqrt 2-1$. It has
   singularities at $\pm \sqrt 2-1$, $\pm 1$ and $\pm i$, and admits
   an analytic continuation in 
$$\cs\setminus \left([-\sqrt 2 -1, -1] \cup 
[\sqrt 2-1, 1] \cup [i, i\infty) \cup [-i, -i\infty)\right).$$ 
\end{Lemma}
\begin{proof}
  The first two statements are obvious. Now assume that the roots $u$
  and $1/u$ have
  modulus~$1$,  that is, $u=e^{i\theta}$ for
  $\theta\in\rs$. This means that 
$
f(t):=\frac{1-t+t^2+t^3}{2t}=\cos \theta$ is real, and belongs to the interval
$[-1,1]$. Write $t=x+iy$, and express $\Im f(t)$ in terms of $x$ and
$y$. One finds that $f(t)$ is real if and only if
either $y=0$ (that is, $t \in \rs$) or 
\beq\label{critcurve}
y^2(1+2x)=1-{x}^{2}-2\,{x}^{3}.
 \eeq 
Since $y^2\ge 0$,  this is only possible if $-1/2<x \le x_c$ where 
$x_c \sim 0.65...$ satisfies $ 1-{x}_c^{2}-2\,{x}_c^{3}=0$.
Observe that the above curve includes $\cE_0$.

For real values of $t$, an elementary study of $f$ shows that
$f(t)\in [-1,1]$ if and only if $t\in [-\sqrt 2-1, -1]\cup [\sqrt
  2-1,1]$ (see Fig.~\ref{fig:f}, left). 
If $t=x+iy$ is non-real and~\eqref{critcurve} holds, then $f(t)={\frac
  {-1+4\,{x}^{2}+4\,{x}^{3}}{1+2\,x}}$. Given that $-1/2<x \le x_c$, this
belongs to $[-1,1]$ if and only if $x$ is non-negative (see
Fig.~\ref{fig:f}, middle). We have thus proved that $|u|=1$ if and
only if $t \in \cE$.

\smallskip
The properties of the series $U$ follow from  basic complex
analysis. Of course, one may choose the position of the cuts
differently, provided they include the 6 singularities. With the cuts
along the coordinate axes, a plot of the
modulus of $U$ is shown on the right of Fig.~\ref{fig:f}.
\end{proof}

\begin{figure}[htb]
\includegraphics
[scale=0.25]{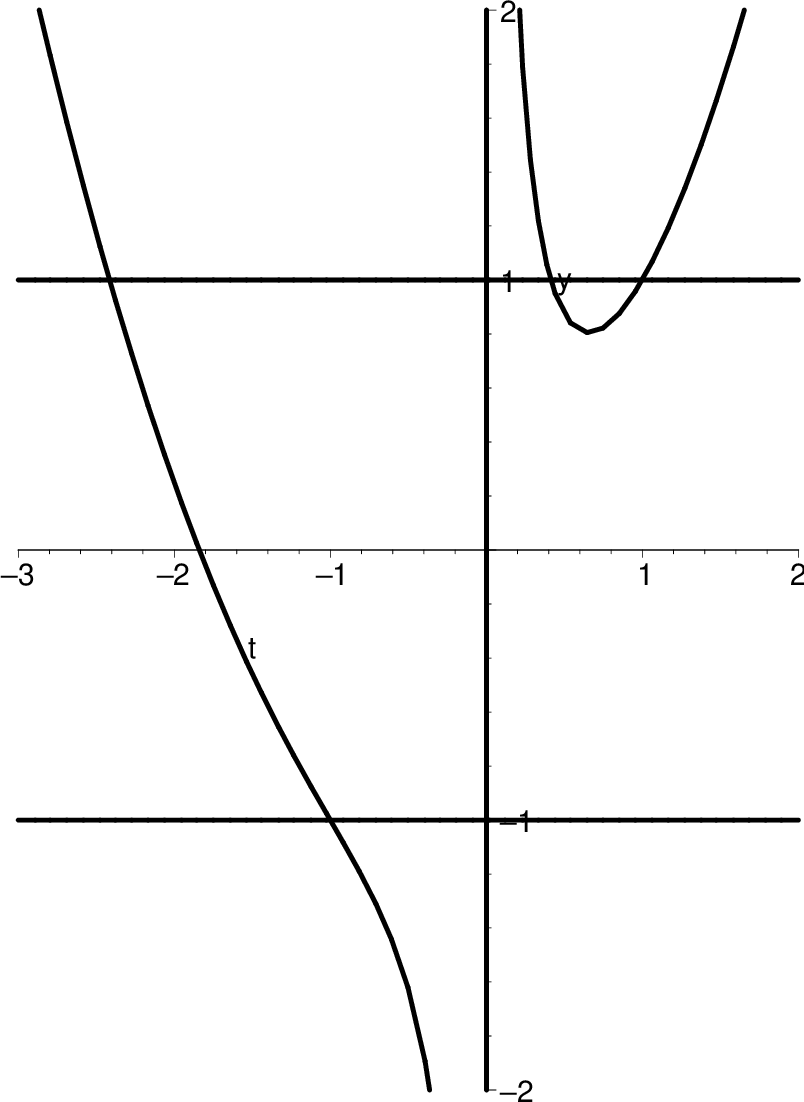}
\hskip 10mm
\includegraphics
[scale=0.25]{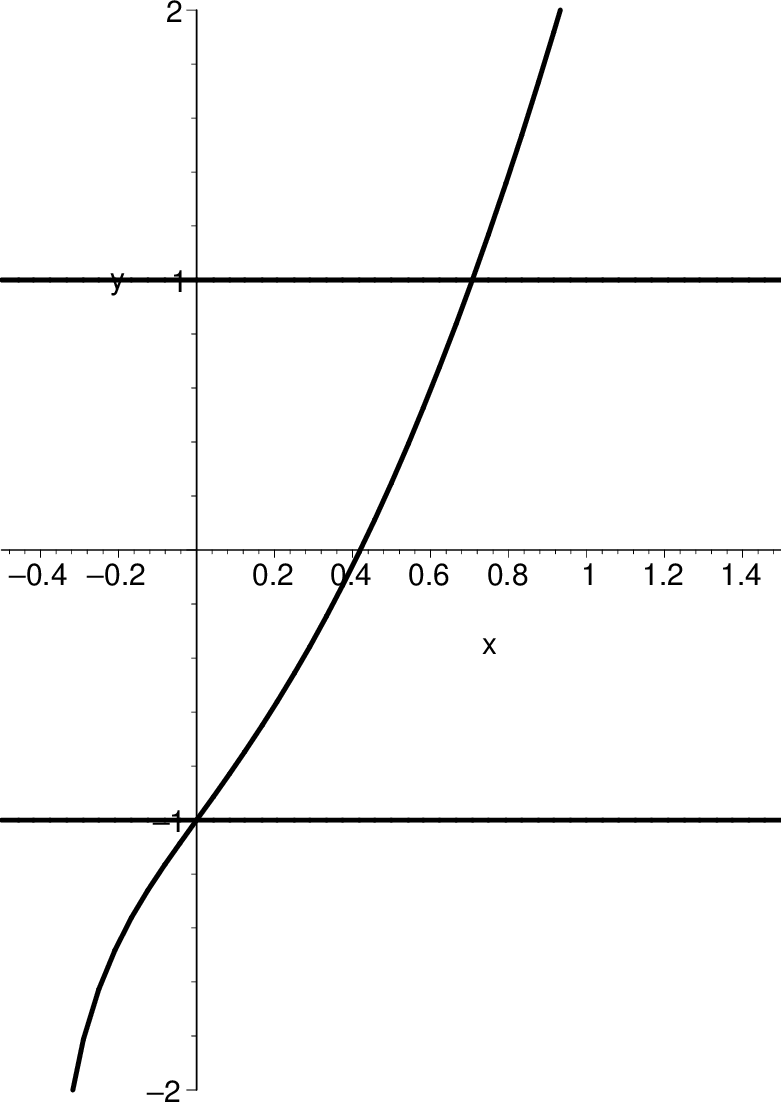}
\hskip 10mm
\includegraphics
[scale=0.3]{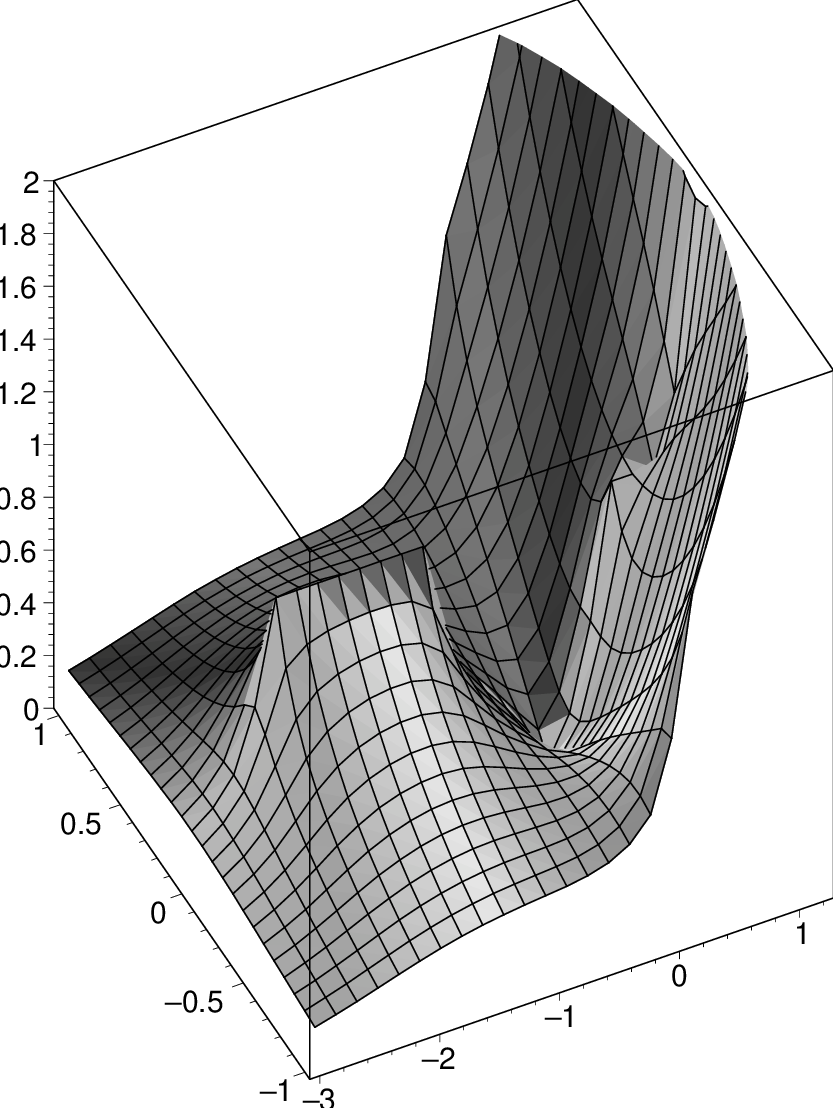}
\caption{The functions $t\mapsto f(t)=\frac{1-t+t^2+t^3}{2t}$,  $x\mapsto {\frac
  {-1+4\,{x}^{2}+4\,{x}^{3}}{1+2\,x}}$, and a plot of the modulus of
  $U$, showing the two cuts on the real axis.}
\label{fig:f}
\end{figure}

\begin{Lemma}\label{lem:acc}
Let $\cE$ be the subset of $\cs$ defined in
Proposition~{\rm\ref{prop:nature}}, and $G_k$ the polynomials of Proposition~{\rm\ref{prop:culm-h}}.\\
If $G_k(z)=G_\ell(z)=0$ with  $\ell\not = k$, then $z\in \cE$.\\
If $G_k(z)=0$ and $z$ is non-real, then $z\not \in \cE$.\\
The set of accumulation points of  roots of the  polynomials $G_k$
is  exactly $\cE$.
\end{Lemma}
The latter point is illustrated in Fig.~\ref{fig:crit}.
\begin{proof}
Note first that, for $z\not = 0$, 
$$
{G_k(z)}=z^k\ \frac{\bigl((1-z)u-z\bigr)u^k-\bigl((1-z)/u-z\bigr)u^{-k}}{u-1/u},
$$
where  $u$ and $1/u$ are the
two roots of $z(u+1/u)=1-z+z^2+z^3$. 

Assume  $G_k(z)=G_\ell(z)=0$. Then $z\not =0$ (because
$G_k(0)=1$). The equations  $G_k(z)=G_\ell(z)=0$ imply
$u^{2k}= u^{2\ell}$,
so that $|u|=1$, that is, by Lemma~\ref{lem:U}, $z \in \cE$.

\medskip
Assume  $G_k(z)=0$ and $z$ is non-real. Assume moreover that  $z\in
\cE$. Let $u$ and $1/u$ be defined as above.  By
Lemma~\ref{lem:U},  $|u|=1$. Write  $u= e^{i\theta}$.
Then $G_k(z)=0$ implies
$$
\frac z {1-z}= \frac{\sin((k+1)\theta)}{\sin(k\theta)},
$$
which contradicts the assumption that $z$ is non-real.

\medskip
Now let $z$ be an accumulation point of roots of the $G_k$'s. There
exists a sequence $z_i$ that tends to $z$ such that $G_{k_i}(z_i)=0$,
  with $k_i\rightarrow \infty$. We want
to prove that $z\in\cE$. If $z$ is one of the 6 singularities of $U$,
then there is nothing to prove. Otherwise, $U$ has an
analytic description in a neighborhood of $z$. The equation
$G_{k_i}(z_i)=0$ reads
$$
U(z_i)^{2k_i}= \frac{(1-z_i)/U(z_i)-z_i}{(1-z_i)U(z_i)-z_i}.
$$
By continuity, $U(z_i)\rightarrow U(z)$  as $i\rightarrow \infty$.
If $z=0$, then $U(z)=0$ and the right-hand side diverges while
the left-hand side tends to $0$. This is impossible, and hence $z\not
=0$. This implies that 
the right-hand side tends to a finite, non-zero limit and, by
continuity, forces
 $|U(z)|=1$. By Lemma~\ref{lem:U}, this 
means that $z\in \cE$. 

\medskip
Conversely,   let $z\in \cE$. By Lemma~\ref{lem:U}, the two roots of
  $z(u+1/u)=1-z+z^2+z^3$ can be written $e^{\pm i\theta}$. By density,
we may assume that $\theta=j\pi/\ell$, for $0<j<\ell$. 
This excludes in particular the 6 singular
points of $U$, for which $u=\pm 1$. This means
that $U$ has an analytic description in a neighborhood of $z$, so that
for $t$ close to $z$,
$
U(t)=U(z) ( 1 + s)
$
with 
$s= (t-z) \frac{U'(z)}{U(z)}+ O((t-z)^2)$.
Thanks to the equation satisfied by $U(z)$, it is easy to see that
$U'(z)\not = 0 $ if $z \not = x_c$, where $x_c$ is defined in the
proof of Lemma~\ref{lem:U}.  We assume from now on that $z \not = x_c$
(again, by density, this is a harmless assumption). Let $k$ be a
multiple of $\ell$. The
equation $G_k(t)=0$ reads 
$$
U(t)^{2k}= \frac{(1-t)/U(t)-t}{(1-t)U(t)-t},
$$
or, given that $U(z)^{2k}=e^{2ijk \pi/\ell}=1$,
$$
 ( 1 + s)^{2k} = \frac{(1-z)/U(z)-z}{(1-z)U(z)-z} + O(t-z).
$$
The right-hand side being finite and non-zero,
one finds a root $t$ of $G_k$ in the
neighborhood of $z$:
$$
t=z+ \frac{U(z)}{2k U'(z)} \log\left(
\frac{(1-z)/U(z)-z}{(1-z)U(z)-z}\right)
+ o(1/k),
$$
and this root gets closer and closer to $z$ as $k$ increases. Thus $z$
is an accumulation point of roots of the $G_k$'s.
 \end{proof}

\noindent {\em Proof of Proposition~}\ref{prop:nature}.\\
One has $\PB(t)=\sum t^k/G_k(t)$, with
\beq\label{CkU}
\frac{t^k}{G_k(t)}=
\frac{u-1/u}{\bigl((1-t)u-t\bigr)u^k-\bigl((1-t)/u-t\bigr)u^{-k}},
\eeq
$u$ and $1/u$ being the roots of $t(u+1/u)=1-t+t^2+t^3$. Let us first
prove that this series defines a meromorphic function in $\cs\setminus
\cE$.
Assume $t\not
\in \cE$. By Lemma~\ref{lem:acc}, $t$ is not an accumulation point of
roots of the polynomials $G_k$, and cancels at most one of these
polynomials. Hence there exists a neighborhood of 
$t$ in which at most one  of the $G_k$'s has a zero,
which is  $t$ itself.  Moreover, by Lemma~\ref{lem:U}, one of the
roots $u$ and $1/u$ has modulus larger than $1$. By continuity,  this holds in
a (possibly smaller) neighborhood of $t$. Then~\eqref{CkU} shows that the
series $\sum t^k/G_k(t)$ is meromorphic in the vicinity of $t$. Given
that $\cs\setminus \cE$ is connected, we have proved that this series
defines a meromorphic function in $\cs\setminus \cE$. 
The same holds for $W(t)$, which is a rational function of $t$ and
$\PB(t)$. 

\medskip
Let us now prove that $\cE_0$ is a natural boundary of $\PB$.
By Lemma~\ref{lem:acc},  every non-real zero of
$G_k$ is in  $\cs\setminus \cE$, and does not cancel any other
polynomial $G_\ell$. Hence it is a pole of $\PB$. 
Now let $z \in \cE_0$ with $z\not \in \rs$. By 
Lemma~\ref{lem:acc},  this point is an accumulation point of
(non-real) zeroes of the polynomials $G_k$, and thus an accumulation
point of poles of $\PB$.  Thus it is a singularity of $\PB$, and the whole
curve  $\cE_0$ is a natural boundary of $\PB$. Given that
$\PB$ and $W$ are related by a simple homography, this curve is also a
boundary for $W$.
\qed

\section{The diagonal model}
\label{sec:diag}

 We have defined weakly directed walks in the diagonal model
by requiring
that the portion of the walk joining two visits to the same diagonal
is partially directed. This is analogous to
the definition we had in the horizontal model. The definition of
bridges  is adapted accordingly, by defining the  height of a
vertex as the sum of its coordinates. However, there is no
simple counterpart of Proposition~\ref{prop:equiv}: the irreducible
bridges of a weakly directed bridge may not be partially directed (Fig.~\ref{fig:weak-d}). 
However, it is easy to see that bridges formed of partially directed irreducible bridges are always weakly directed.  
In this section, we enumerate these walks and study their asymptotic
properties.

\subsection{Generating function}
\begin{Proposition}\label{prop:weak-d}
  The \gf\ of bridges formed of partially directed irreducible bridges is
$$
W_\Delta(t)=\frac 1 {1+2t-\displaystyle  \frac{ 2t\PB_1}{1+  t\PB_1}
- \frac{4t \PB_2}{1+ 2t \PB_2}
+\frac{ 2t\PB_0}{1+  t\PB_0}},
$$
where the series $\PB_{i}
=\sum_{k\ge 0} \PB _{i}^{(k) }(t) $ are given
in Propositions~{\rm\ref{prop:culm-d-1}, \ref{prop:culm-d-2},
  \ref{prop:culm-d-0}}.

\end{Proposition}
\begin{proof}
    Let $\cI_\SS$ be the set of irreducible \EE\SS\WW-bridges, and let
  $I_\SS$ be the associated length \gf. Given that a
non-empty
 \EE\SS\WW-bridge is 
obtained by adding an \EE\ step at the end of a  \EE\SS\WW-pseudo-bridge,
and  a non-empty sequence of irreducible \EE\SS\WW-bridges, there holds
$$
t \PB_1= \frac {I_\SS}{1-I_\SS}.
$$
Define  similarly the sets $\cI_\NN$, $\cI_\EE$ and $\cI_\WW$, and the
associated series $I_\NN$, $I_\EE$ and $I_\WW$.  Finally, let
$\cI_{\EE\SS}$ (resp. $\cI_{\NN\WW}$) be the 
set of irreducible \EE\SS-bridges (resp. \NN\WW-bridges), and let
$I_{\EE\SS}$ (resp. $I_{\NN\WW}$) be the associated series. Then
$$
2t \PB_2= \frac {I_\EE}{1-I_\EE}
\quad\hbox{ and } \quad 
t \PB_0= \frac {I_{\EE\SS}}{1-I_{\EE\SS}}.
$$
(The factor 2 comes from the fact that a \NN\EE\SS-bridge may end with
a \NN\ or \EE\ step.)
By symmetry,
$I_\NN=I_\EE$,  $I_\WW=I_\SS$ and $\cI_{\EE\SS}=\cI_{\NN\WW}$. 
Moreover,
$$
\cI_\SS\cap \cI_\NN= \EE,\quad
\cI_\SS\cap \cI_\EE= \cI_{\EE\SS}, \quad
\cI_\SS\cap \cI_\WW= \emptyset, \quad
\cI_\EE \cap  \cI_\NN=  \NN + \EE,\quad
\cI_\EE\cap \cI_\WW=\NN,\quad
\cI_\WW\cap  \cI_\NN=\cI_{\NN\WW}.
$$
By an elementary inclusion-exclusion argument,  the \gf\ of partially directed irreducible bridges is
$$
I:=2I_\SS+2I_\EE -2I_{\EE\SS}- 2t=
 \frac{2t \PB_1}{1+  t\PB_1}
+ \frac{4t \PB_2}{1+ 2t \PB_2}
-\frac{2t \PB_0}{1+  t\PB_0}-2t.
$$
Hence the \gf\ of   bridges formed of partially directed irreducible bridges
is $ W_\Delta=\frac 1{1-I}$. The  result follows.
\end{proof}


\subsection{Asymptotic properties}
We obtain  for the  diagonal model 
 asymptotic results that are similar to those obtained in the horizontal model,
with a slightly smaller growth constant.
We have to confess that this contradicts our original intuition: since
in the horizontal 
model, two of the four classes of irreducible partially directed
bridges (namely, \EE\SS\WW\ and
\NN\EE\WW) are either trivial or degenerate, while in the
diagonal model, all four classes are non-trivial, we thought we had a
 chance to reach a better growth constant in the diagonal model.
This is unfortunately not the case. We  nonetheless
present this diagonal variant, because we believe it to be  a natural
attempt.
 We analyze below what makes the difference between the two growth
 constants, and this analysis shows that our hopes could just as
 well have come true.

\begin{Proposition}\label{prop:asympt-d}
The \gf\ $W_\Delta$ 
given by
Proposition~{\rm\ref{prop:weak-d}} is meromorphic in the disk 
  $\cD=\{|z|< \sqrt 2 -1\}$. It has a
  unique dominant pole in this disk,  at $\rho_1\simeq 0.3940$.
This pole is simple. Consequently,  the
  number of $n$-step bridges formed of  partially directed
irreducible bridges 
is asymptotically equivalent to $ \kappa\, {\mu}^n 
$,
with $\mu=1/\rho_1\simeq 2.5378$.

Let $N_n$ denote the number  of irreducible bridges  in a random
$n$-step bridge formed of  partially directed
irreducible bridges. The mean and
variance of $N_n$ satisfy:
$$
\E(N_n) \sim  \mathfrak m  \, n,  \quad  \quad
\Var(N_n) \sim  \mathfrak s ^2 \, n ,  
$$
where 
$$
 \mathfrak m  \simeq 0.395 \ \ 
\quad \hbox{and} \quad
\mathfrak s ^2 = 1\pm 2.10^{-3},
$$
 and the random variable
$
\frac{N_n-  \mathfrak m \, n}{\mathfrak s \sqrt n}
$
converges in law to  a standard normal distribution. In particular,
the average end-to-end distance, being bounded from below by
$\E(N_n)$, grows linearly with~$n$.
\end{Proposition}
\begin{proof}
 The arguments are  the same as in the proof of
Proposition~\ref{prop:asympt-h}. This series reads $W_\Delta=1/(1-I)$, where
$I$ counts partially directed irreducible bridges. The only change is
in the bounds we use on the series $I$:
$$
I^-(t) \le I(t) \le I^+(t),
$$
with
$$
I^-(t):= I_{\le n}(t)
\quad \hbox{ and } \quad 
I^+(t):= I_{\le n}(t) + 4T_{>n}(t)=  I_{\le n}(t) + 4T(t)- 4T_{\le n}(t),
$$
where
$T(t)$ is the \gf\ of \NN\EE\SS-walks, given in Lemma~\ref{lem:TPQ}.
%
\end{proof}

\noindent{\bf Remark.}  Hence the growth constant in the diagonal
model is a bit   smaller than in the horizontal model. This does not
seem to be predictible. The series of Propositions~\ref{prop:weakly-h}
and~\ref{prop:weak-d} respectively read
$$
W(t)=\frac 1 {1-I(t)} \quad \hbox{and} \quad W_\Delta(t)=\frac 1 {1-I_\Delta(t)} 
$$
where $I$ and $I_\Delta$ count irreducible partially directed bridges,
respectively in the horizontal and diagonal model. As $t$ increases
from $0$ to $\sqrt 2 -1$ (the radius of convergence of the series of
partially directed walks), $I_\Delta(t)=2t+O(t^2)$ first dominates
$I(t)=t+O(t^2)$, but the graphs of these two functions cross
before any of them reaches 1 (Fig.~\ref{fig:hor-diag}), so that $I(t)$ reaches 1
before $I_\Delta(t)$ does. The fact that the graphs cross
is consistent with our belief that $I$ has radius $\sqrt 2 -1\sim 0.41$ while
$I_\Delta$ has a larger radius of convergence, namely $1/\sqrt
5\sim 0.44$. But $I_\Delta(t)$  could just as well have reached 1 before
the crossing point.
\begin{figure}[ht]
\includegraphics[scale=0.2]{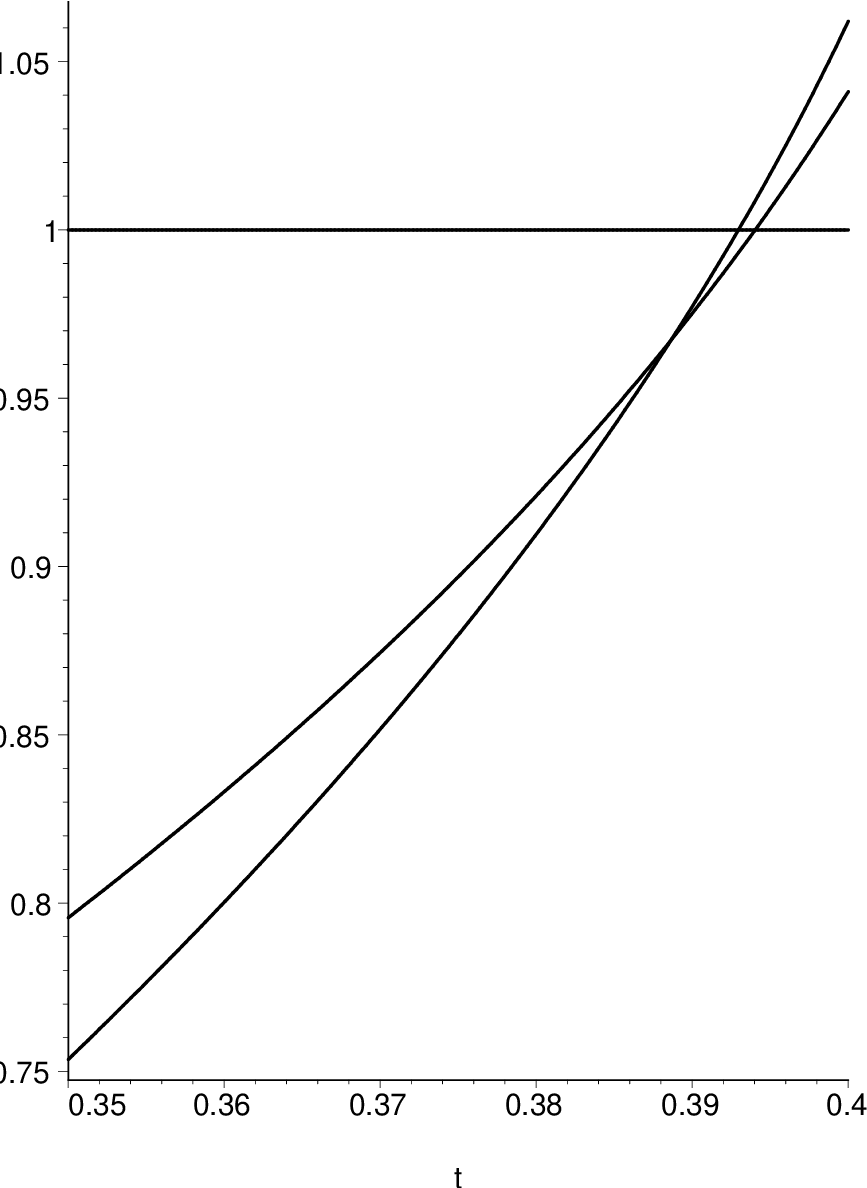}
\caption{The functions $I(t)$ and $I_\Delta(t)$ for $t \in
  (0.35,0.4)$. The function $I_\Delta(t)$ first dominates, but the
  graphs cross before the functions reach 1.}
\label{fig:hor-diag}
\end{figure}
\section{Final comments}
\label{sec:comments}
\subsection{One more way to count partially directed bridges}
We have presented in Sections~\ref{sec:culminant-kernel}
and~\ref{sec:culminant-heaps} two approaches to 
count partially directed bridges. Here, we discuss a third method,
 based on standard decompositions of lattice paths.
This approach involves some
guessing, whereas the two others don't. 
In the diagonal model,
it allows us to understand more combinatorially why the \gfs\ of \NN\EE\SS- and
\EE\SS\WW-bridges only differ by a factor $(2-t^2)^k$. 
Moreover, this approach is needed for random generation.

\medskip

We first discuss the horizontal model, that is, the enumeration of
\NN\EE\SS-bridges given by
Proposition~\ref{prop:culm-h}. 
We  say that a \NN\EE\SS-walk is an \emm
excursion, if it starts and ends at height $0$, and all its vertices
lie at a non-negative height. Let $E^{(k)}\equiv E^{(k)}(t)$ be the
length \gf\ of excursions of height at most $k$. As before, $\PB^{(k)}$
denotes the \gf\ of \NN\EE\SS-pseudo-bridges of height $k$.

\begin{figure}[ht]
\includegraphics[scale=0.7]{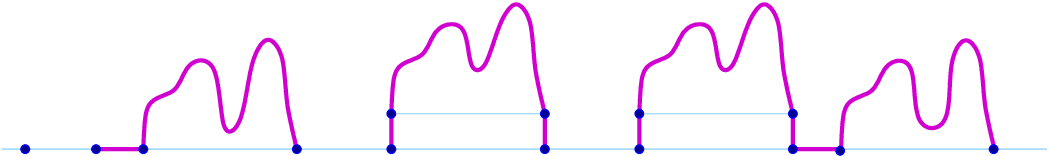}
\caption{Recursive decomposition of \NN\EE\SS-excursions and pseudo-bridges.}\label{fig:decomp-excursions}
\end{figure}

Excursions and pseudo-bridges can be factored in a standard way by
cutting them at their first (resp. last) visit at height $0$. These
factorizations are schematized in Fig.~\ref{fig:decomp-excursions}. They give:
\begin{itemize}
\item for excursions of bounded height, the recurrence relation
$$
E^{(k)}=1+ t E^{(k)} + t^2 \left( E^{(k-1)}-1\right) +  t^3 \left(
E^{(k-1)}-1\right)  E^{(k)},
$$
with the initial condition $E^{(-1)}=1$,
\item for pseudo-bridges of height $k$, the recurrence relation
$$
\PB^{(k)}= \left (1 + tE^{(k)}\right) t \PB^{(k-1)},
$$
with the initial condition $\PB^{(0)}=1/(1-t)$.
\end{itemize}
It is then straightforward to \emm check, by induction on $k$ that
$$
E^{(k)}= \frac 1 t \left( \frac{G_{k-1}}{G_k}-1\right)
\quad \hbox{ and  }\quad \PB^{(k)}= \frac{t^k}{G_k},
$$
where $G_k$ is the sequence of polynomials defined in
Proposition~\ref{prop:culm-h}. Of course, these expressions have to be
guessed --- which is actually not very difficult using a
computer algebra system.

\bigskip

Let us now discuss  the diagonal model. 
We will recover Proposition~\ref{prop:culm-d-1} (for \EE\SS\WW-bridges, or,
equivalently, \NN\SS\WW-bridges) and Proposition~\ref{prop:culm-d-2}
(for \NN\EE\SS-bridges).
%
 Let $E_1^{(k)}$ and $E_2^{(k)}$ be the generating functions of \NN\SS\WW- and
\NN\EE\SS-excursions, respectively, of height at most~$k$.  These
two series coincide: indeed, a \NN\EE\SS-excursion is
obtained by reading  a \NN\SS\WW-excursion {backwards}, reversing
 each step, and  this  does not change the height of the excursion.

\begin{figure}[ht]
\includegraphics[scale=0.7]{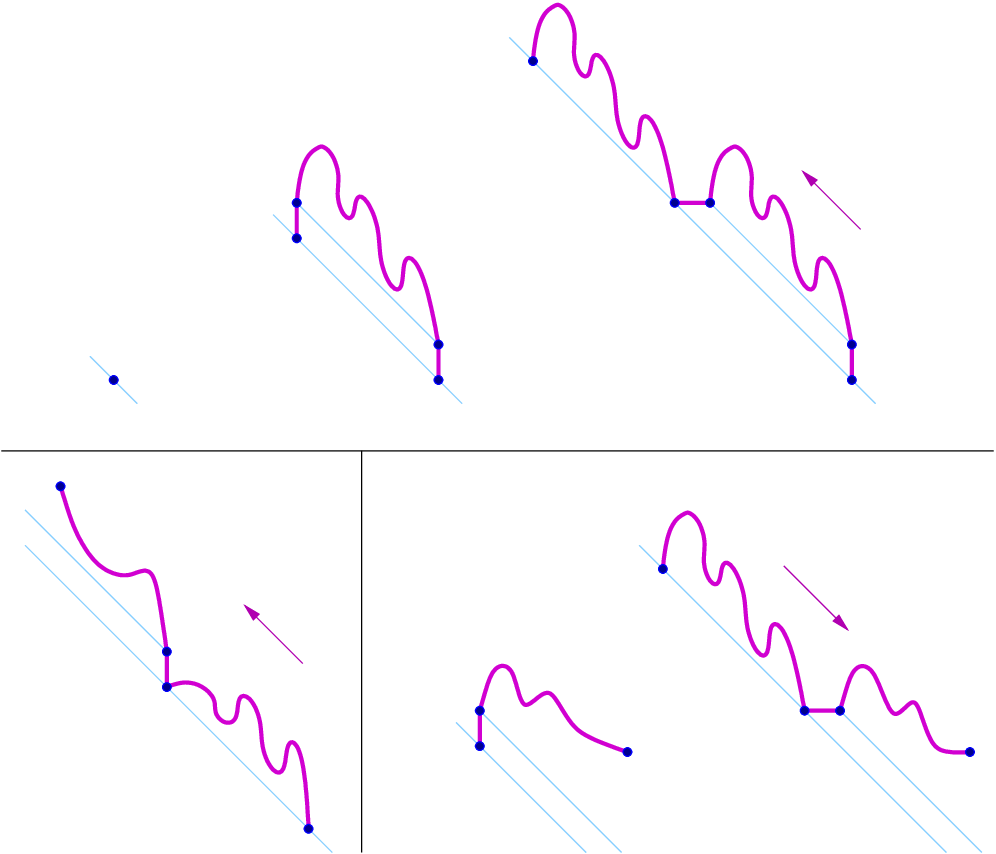}
\caption{\emm Top,: Recursive decomposition of \NN\SS\WW-excursions (or
  equivalently, of \NN\EE\SS-excursions).
\emm Bottom,: Recursive decomposition of \NN\SS\WW-pseudo-bridges (left)
and \NN\EE\SS-pseudo-bridges (right).}
\label{fig:decomp-excursions-d}
\end{figure}

Partially directed excursions and pseudo-bridges can be factored as in
the horizontal case, as illustrated in
Fig.~\ref{fig:decomp-excursions-d}. This gives:
\begin{itemize}
\item 
for \NN\SS\WW-excursions, the recurrence relation
\[E_1^{(k)} = 1 + t^2\bigl(E_1^{(k-1)} - 1\bigr) +
t^2E_1^{(k-1)}E_1^{(k)},\]
with the initial condition $E_1^{(0)}=1$,
\item for \NN\SS\WW-pseudo-bridges,
\beq\label{psb1}
B_1^{(k)} = D_1^{(k)}tB_1^{(k-1)},
\eeq
where  $D_1^{(k)}$ counts \NN\SS\WW-excursions of height at most $k$ \emph{not
ending with a \SS\  step}, and $B_1^{(0)}=1$,
\item  for \NN\EE\SS-pseudo-bridges,
\beq\label{psb2}
B_2^{(k)}=\bigl(1 + E_1^{(k)}\bigr)tB_2^{(k-1)},
\eeq
with the initial condition $B_2^{(0)} = 1$.
\end{itemize}
As in the horizontal model, it is easy to check by induction that
$$
E_1^{(k)}= E_2^{(k)}= 
(2-t^2)\frac{G_{k-1}}{G_k} -1
\quad \hbox{ and } \quad 
B_2^{(k)} = \frac {(2-t^2)^k t^k}{G_k},
$$
where $G_k$ is the sequence of polynomials defined in
Proposition~\ref{prop:culm-d-1}.
This gives Proposition~\ref{prop:culm-d-2}.

In order to prove Proposition~\ref{prop:culm-d-1}, we will  prove
combinatorially that $1 + E_1^{(k)} = 
(2-t^2)D_1^{(k)}$. By comparing~\eqref{psb1} and~\eqref{psb2}, this
will establish the link $B_2^{(k)}=(2-t^2)^k
B_1^{(k)}$ between the two types of pseudo-bridges.

First, let $u$ be a \NN\SS\WW-excursion of height at most $k$, ending with
\NN\WW. Writing $u=v\NN\WW$, we see that $v$ is an excursion that does not end
with a \SS\  step; therefore,  excursions ending with
\NN\WW\ are counted by $t^2D_1^{(k)}$.

Let now $u$ be an arbitrary \NN\SS\WW-excursion of height at most $k$. We
distinguish two cases:
\begin{itemize}
\item  either $u$ does not end with a \SS\  step; such excursions are
counted by $D_1^{(k)}$;
\item  or $u$ reads $v\SS$; then $v$ does not end with a \NN\ 
step. Let $u'=v\WW$: then $u'$ is an excursion ending with \WW\  but not with
\NN\WW. According to the above remark, such excursions are counted by
$D_1^{(k)} - 1 - t^2D_1^{(k)}$.
\end{itemize}
Putting this together, we find  $E_1^{(k)} = (2-t^2)D_1^{(k)} - 1$,
which concludes the proof.

\subsection{Random generation of weakly directed bridges}

  \begin{minipage}[b]{0.7\linewidth}
We now present an algorithm for the random generation of weakly
directed
bridges in the horizontal model.
This algorithm is a \emph{Boltzmann sampler}~\cite{duchon}. 
That is, it involves a parameter $x$, and outputs a walk $w$ with probability
\[\mathbb P(w)=\frac{x^{\lvert w\rvert}}{C(x)},\]
where $C(x)$ is the generating function of the 
class of walks under consideration. Of course, $x$ has to be smaller
than the radius of convergence of $C$. 
 The average length of the output walk is
$$
\E(|w|)= \frac{x C'(x)}{C(x)}.
$$
The parameter $x$ 
is chosen according to the desired  output length.

Boltzmann samplers have  convenient properties. For instance, given
Boltzmann samplers $\Gamma_{\mathcal A}$ and $\Gamma_{\mathcal B}$ for
two classes $\mathcal A$ and $\mathcal B$, it is easy to derive  
Boltzmann samplers for the classes $\mathcal A+\mathcal B$ 
(assuming $\mathcal A\cap\mathcal B=\varnothing$)
and
$\mathcal A\times\mathcal B$. In the former case, one calls
$\Gamma_{\mathcal A}$ with probability $A(x)/(A(x)+B(x))$, and
$\Gamma_{\mathcal B}$ with probability $B(x)/(A(x)+B(x))$. In the
latter case, the sampler is just $(\Gamma_{\mathcal A},\Gamma_{\mathcal B})$.  
If the base samplers run in linear time with
respect to the size of the output, the new samplers also run in linear time.

Moreover, if $\mathcal B\subseteq\mathcal A$, and  we 
have  a Boltzmann sampler for $\mathcal A$,  then a
\emph{rejection scheme} 
provides  a Boltzmann sampler for
$\mathcal B$: we keep drawing elements of
$\mathcal A$ until we find an element of $B$.

Finally, if
$\mathcal A=\mathcal B\times\mathcal C$, then we obtain a Boltzmann
sampler for the class $\mathcal B$ by sampling a pair $(b,c)$ and
discarding $c$.

\bigskip
\begin{center}
  \includegraphics[height=25mm]{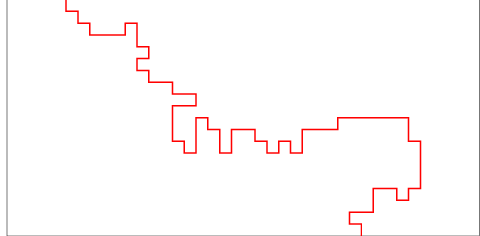}

{\sc Figure 14.} Right: A random weakly directed bridge drawn with our
algorithm. Above: zoom on a portion of the bridge.
\end{center}
\end{minipage}
 \begin{minipage}[b]{0.30\linewidth}

\ \hskip 10mm \includegraphics[height=15cm]{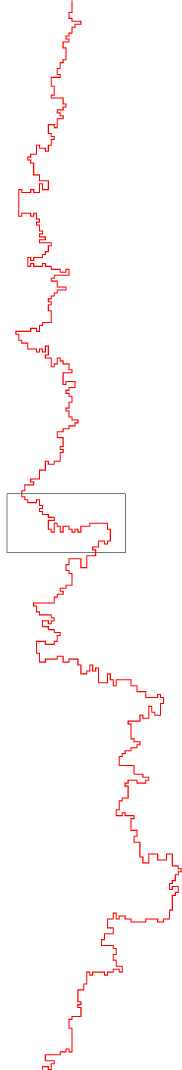}
\end{minipage}

\bigskip

By Proposition~\ref{prop:equiv}, a weakly directed bridge is
 a sequence of partially directed irreducible bridges. 
We build our algorithm in four steps, in which we sample objects of
increasing complexity.

\begin{description}
\item[Step 1] The first step is to sample partially directed excursions. Let
$\mathcal E$ be the language of nonempty \NN\EE\SS-excursions. As shown
  by Fig.~\ref{fig:decomp-excursions}, this language is determined by the
following unambiguous grammar:
\[\mathcal E = \EE(1+\mathcal E) + \NN\mathcal E\SS
+ \NN\mathcal E\SS\EE(1+\mathcal E).\]
We use this grammar to derive, first the generating function $E(x)$ of partially directed excursions:
$$
E(x)=\frac{1-x-x^2-x^3-\sqrt{(1-x^4)(1-2x-x^2)}}{2x^3},
$$
 and then a (recursive) Boltzmann sampler for these excursions
(see~\cite[Section~3]{duchon}).

\item[Step 2] The next step is to sample positive 
\NN\EE\SS-walks, defined in
Section~\ref{sec:gf1}. 
More precisely, let $\mathcal P_\NN$ be the language of positive
\NN\EE\SS-walks that end with a \NN\  step\footnote{One could just as well work with \emm general, positive walks, but it can be seen that this restriction  improves the complexity by a constant factor.}.
We decompose these walks in the same way as bridges in
Fig.~\ref{fig:decomp-excursions}, by cutting them after their last visit at
height~$0$. 
We have the following unambiguous grammar: 
\[\mathcal P_\NN = (\NN+\EE\NN+\mathcal E\EE\NN)(1+\mathcal P_\NN).\]
Given the Boltzmann sampler constructed for excursions in Step~1, we thus obtain a Boltzmann sampler 
for these positive walks. Their \gf\  is
\[P_\NN(x)=\frac12\Biggl(\sqrt{\frac{1+x+x^2+x^3}{(1-x)(1-2x-x^2)}}-1\Biggr).\]

\item[Step 3] The object of this step is to sample irreducible
\NN\EE\SS-bridges; this is  less routine
 than the two previous steps, as we do not have a grammar for these walks. To
do this, we decompose the walks of $\mathcal P_\NN$ into irreducible factors (see
Definition~\ref{def:irreducible}). Let 
$\mathcal I_\EE$ be the language of irreducible
\NN\EE\SS-bridges, and $\mathcal R$  the language of
irreducible positive \NN\EE\SS-walks ending with a \NN\  step 
that  are not
bridges. Performing the decomposition and checking whether the first factor is
a bridge or not, we find:
\[\mathcal P_\NN = \mathcal R + \mathcal I_\EE(1+\mathcal P_\NN).\]
We  use this to construct a Boltzmann sampler for irreducible
\NN\EE\SS-bridges: Using a rejection scheme, we first derive from
the Boltzmann sampler of $\cP_\NN$  a Boltzmann
sampler for the language $\mathcal I_\EE(1+\mathcal P_\NN)$;  these walks factor  into
an irreducible bridge, followed by a positive walk. We then simply discard the
latter walk, keeping only the irreducible bridge.

We construct symmetrically a Boltzmann sampler for the language $\mathcal
I_\WW$ of irreducible \NN\SS\WW-bridges. We use
another rejection scheme to sample elements of $\mathcal I_\WW\setminus\NN$.

\item[Step 4] Finally, the language $\mathcal W$ of weakly directed bridges
satisfies, 
as explained in the proof of Proposition~\ref{prop:weakly-h}:
\[\mathcal W = 1 + \mathcal I_\EE\mathcal W + (\mathcal I_\WW\setminus\NN)\mathcal W.\]
From this, we obtain a Boltzmann sampler for weakly directed bridges.
\end{description}

\begin{Proposition}
Let $\varepsilon$ be a fixed positive real number. The random generator
described above, with the parameter $x$ chosen such that $xW'(x)/W(x)=n$,
outputs a weakly directed bridge with a length between
$(1-\varepsilon)n$ and $(1+\varepsilon)n$ in average time $\mathrm O(n)$.
\end{Proposition}
\begin{proof}
Let $x>0$ be  smaller than the radius of convergence $\rho$ of the generating
function $W$, given by Proposition~\ref{prop:asympt-h}.
We first prove that if our algorithm outputs a walk of length $m$, it has, on
average, run in time $\mathrm O(m)$, independently of the parameter $x$.

The radius of convergence of the generating function $P_\NN(x)$ 
is  $\sqrt2-1$,
and is therefore larger than $\rho$. Hence the average length of a
positive walk drawn according to the Boltzmann distribution of
parameter $x$, being $xP_\NN'(x)/P_\NN(x)$,
is bounded from above by  $\rho P_\NN'(\rho )/P_\NN(\rho )$ 
(see~\cite[Prop.~2.1]{duchon}), which is 
 {independent of  $x$}. 
In particular, the
algorithm described in Step~2 runs in average constant time
(and the average length of the output walk is bounded).

Testing whether a positive walk is in $\mathcal R$ can be done in linear time.
Moreover, the probability of success in Step~3 is
$$
I_\EE(x)\left(1+1/P_\NN(x)\right)\ge x\left(1+1/P_\NN(x)\right)= \frac 1 2 \left( \sqrt{(1-x^4)(1-2x-x^2)} +1-x+x^2+x^3\right)\ge \sqrt 2 -1.
$$
Thus the average number of trials necessary to draw a
 walk of $\cI_\EE(1+\cP_\NN)$ is bounded by a constant independent of $x$.
Therefore, the algorithm that outputs walks of $\cI_\EE$ also runs in
average constant time. 
 The probability to draw in this step a walk of
$\EE\NN(1+\cP_\NN)$ is $x^2\left(1+1/P_\NN(x)\right)$, which is
bounded from below by $x( \sqrt 2 -1)$. Since in practice 
$x$ will be away from $0$, the probability to obtain an element of $\cI_\EE$
distinct from \NN\ is bounded from below by a positive constant, so
that we generate a walk of $\cI_\EE\setminus\NN$ (or, symmetrically,
of $\cI_\WW\setminus\NN$) in average constant time.

Finally, the number of irreducible bridges in a weakly directed 
bridge  of length $m$ is less than $m$, so that the final algorithm runs in
average time $\mathrm O(m)$.

\bigskip

We now fix $n$ and $\varepsilon$, and choose $x$ as described above (this
is possible since $xW'(x)/W(x) \rightarrow \infty$ as $x\rightarrow
\rho$). We call our sampler of bridges until the length $m$ of the
output bridge is
in the required interval. Theorem~6.3 in~\cite{duchon} implies that,
asymptotically in $n$, a bounded number of trials will
suffice. Indeed, the series $W(t)$ is analytic in a $\Delta$-domain,
with a singular exponent~$-1$ (see Proposition~\ref{prop:asympt-h}).
\end{proof}

Fig.~14
shows a weakly directed 
bridge sampled using our algorithm, with a zoom on a portion of it.



\bibliographystyle{plain}
\bibliography{saw.bib}

\begin{thebibliography}{10}

\bibitem{alm-janson}
S.~E. Alm and S.~Janson.
\newblock Random self-avoiding walks on one-dimensional lattices.
\newblock {\em Comm. Statist. Stochastic Models}, 6(2):169--212, 1990.

\bibitem{bacher-mbm-fpsac}
A.~Bacher and M.~Bousquet-M\'elou.
\newblock Weakly directed self-avoiding walks.
\newblock In {\em FPSAC 2010}, DMTCS Proceedings, pages 473--484, 2010.

\bibitem{hexacephale}
C.~Banderier, M.~Bousquet-M{\'e}lou, A.~Denise, P.~Flajolet, D.~Gardy, and
  D.~Gouyou-Beauchamps.
\newblock Generating functions for generating trees.
\newblock {\em Discrete Math.}, 246(1-3):29--55, 2002.

\bibitem{mbm-prudent}
M.~Bousquet-M\'elou.
\newblock Families of prudent self-avoiding walks.
\newblock {\em J. Combin. Theory Ser. A}, 117(3):313--344, 2010.
\newblock Arxiv:0804.4843.

\bibitem{bousquet-petkovsek-1}
M.~Bousquet-M{\'e}lou and M.~Petkov{\v{s}}ek.
\newblock Linear recurrences with constant coefficients: the multivariate case.
\newblock {\em Discrete Math.}, 225(1-3):51--75, 2000.

\bibitem{mbm-ponty}
M.~Bousquet-M\'elou and Y.~Ponty.
\newblock Culminating paths.
\newblock {\em Discrete Math. Theoret. Comput. Sci.}, 10(2), 2008.
\newblock ArXiv:0706.0694.

\bibitem{brydges-slade}
D.~Brydges and G.~Slade.
\newblock Renormalisation group analysis of weakly self-avoiding walk in
  dimensions four and higher.
\newblock In {\em Proceedings of the International Congress of Mathematicians},
  Hyderabad, India, 2010.
\newblock Arxiv:1003.4484.

\bibitem{de-gennes}
P.~G. de~Gennes.
\newblock Exponents for the excluded volume problem as derived by the {W}ilson
  method.
\newblock {\em Phys. Lett. A}, 38(5):339--340, 1972.

\bibitem{guttmann-prudent}
J.~C. Dethridge and A.~J. Guttmann.
\newblock Prudent self-avoiding walks.
\newblock {\em Entropy}, 8:283--294, 2008.

\bibitem{duchi}
E.~Duchi.
\newblock On some classes of prudent walks.
\newblock In {\em FPSAC'05}, Taormina, Italy, 2005.

\bibitem{duchon}
Ph. Duchon, Ph. Flajolet, G.~Louchard, and G.~Schaeffer.
\newblock Boltzmann samplers for the random generation of combinatorial
  structures.
\newblock {\em Combin. Probab. Comput.}, 13(4-5):577--625, 2004.

\bibitem{duminil-smirnov}
H.~Duminil-Copin and S.~Smirnov.
\newblock The connective constant of the honeycomb lattice equals
  $\sqrt{2+\sqrt2}$.
\newblock ArXiv:1007.0575, 2010.

\bibitem{duplantier-kostov}
B.~Duplantier and I.~K. Kostov.
\newblock Geometrical critical phenomena on a random surface of arbitrary
  genus.
\newblock {\em Nucl. Phys. B}, 340(2-3):491--541, 1990.

\bibitem{flajolet-sedgewick}
P.~Flajolet and R.~Sedgewick.
\newblock {\em Analytic combinatorics}.
\newblock Cambridge University Press, Cambridge, 2009.

\bibitem{foata}
D.~Foata.
\newblock {A noncommutative version of the matrix inversion formula}.
\newblock {\em Adv. Math.}, 31:330--349, 1979.

\bibitem{guttmann-conway2001}
A.~J. Guttmann and A.~R. Conway.
\newblock {Square lattice self-avoiding walks and polygons}.
\newblock {\em Ann. Comb.}, 5(3-4):319--345, 2001.

\bibitem{guttmann-wormald-spiral}
A.~J. Guttmann and N.~C. Wormald.
\newblock On the number of spiral self-avoiding walks.
\newblock {\em J. Phys. A: Math. Gen.}, 17:L271--L274, 1984.

\bibitem{hammersley-welsh}
J.~M. Hammersley and D.~J.~A. Welsh.
\newblock {Further results on the rate of convergence to the connective
  constant of the hypercubical lattice.}
\newblock {\em Q. J. Math., Oxf. II. Ser.}, 13:108--110, 1962.

\bibitem{hopcroft}
J.~E. Hopcroft, R.~Motwani, and J.~D. Ullman.
\newblock {\em Introduction to Automata Theory, Languages and Computation}.
\newblock Addison-Wesley, 3rd edition, 2006.

\bibitem{jensen-bridges}
I.~Jensen.
\newblock Improved lower bounds on the connective constants for two-dimensional
  self-avoiding walks.
\newblock {\em J. Phys. A: Math. Gen.}, 37(48):11521--11529, 2004.

\bibitem{jensen-guttmann}
I.~Jensen and A.~J. Guttmann.
\newblock Self-avoiding polygons on the square lattice.
\newblock {\em J. Phys. A}, 32(26):4867--4876, 1999.

\bibitem{kennedy-pivot-SAW}
T.~Kennedy.
\newblock A faster implementation of the pivot algorithm for self-avoiding
  walks.
\newblock {\em J. Statist. Phys.}, 106(3-4):407--429, 2002.

\bibitem{kesten-bridges}
H.~Kesten.
\newblock {On the number of self-avoiding walks.}
\newblock {\em J. Math. Phys.}, 4(7):960--969, 1963.

\bibitem{lawler-schramm-werner}
G.~F. Lawler, O.~Schramm, and W.~Werner.
\newblock On the scaling limit of planar self-avoiding walk.
\newblock In {\em Fractal geometry and applications: a jubilee of Beno\^\i t
  Mandelbrot, Part 2}, volume~72 of {\em Proc. Sympos. Pure Math.}, pages
  339--364. Amer. Math. Soc., Providence, RI, 2004.

\bibitem{madras-slade}
N.~Madras and G.~Slade.
\newblock {\em The self-avoiding walk}.
\newblock Probability and its Applications. Birkh\"auser Boston Inc., Boston,
  MA, 1993.

\bibitem{nienhuis82}
B.~Nienhuis.
\newblock Exact critical point and critical exponents of {$\mathrm O(n)$}
  models in two dimensions.
\newblock {\em Phys. Rev. Lett.}, 49(15):1062–1065, 1982.

\bibitem{privman-spiral}
V.~Privman.
\newblock Spiral self-avoiding walks.
\newblock {\em J. Phys. A: Math. Gen.}, 16(15):L571--L573, 1983.

\bibitem{prodinger}
H.~Prodinger.
\newblock The kernel method: a collection of examples.
\newblock {\em S\'em. Lothar. Combin.}, 50:Art. B50f, 19 pp. (electronic),
  2003/04.

\bibitem{rechni-buks-saw}
A.~Rechnitzer and E.~J. Janse~van Rensburg.
\newblock {Canonical {M}onte {C}arlo determination of the connective constant
  of self-avoiding walks.}
\newblock {\em J. Phys. A, Math. Gen.}, 35(42):L605--L612, 2002.

\bibitem{stanley-vol1}
R.~P. Stanley.
\newblock {\em Enumerative combinatorics. {V}ol. 1}, volume~49 of {\em
  Cambridge Studies in Advanced Mathematics}.
\newblock Cambridge University Press, Cambridge, 1997.

\bibitem{viennot}
X.~G. Viennot.
\newblock {\em Heaps of pieces, I : Basic definitions and combinatorial
  lemmas}, volume 1234/1986, pages 321--350.
\newblock Springer Berlin / Heidelberg, 1986.

\bibitem{zeilberger-skinny}
D.~Zeilberger.
\newblock Symbol-crunching with the transfer-matrix method in order to count
  skinny physical creatures.
\newblock {\em Integers}, pages A9, 34pp. (electronic), 2000.

\end{thebibliography}

\end{document}